\documentclass[11pt]{article}

\usepackage{fullpage,amsmath,amssymb,color}
\usepackage{algorithm}
\usepackage{enumerate}
\usepackage{algorithmic}
\newcommand{\BlackBox}{\rule{1.5ex}{1.5ex}}
\newenvironment{proof}{\par\noindent{\bf Proof\ }}{\hfill\BlackBox\\[2mm]}

\usepackage{graphicx}
\usepackage{epstopdf}
\usepackage{subfigure}
\graphicspath{{figures/}}
\usepackage{graphicx}   % need for figures
\usepackage{multirow}
\usepackage{amsmath}
\usepackage{amsfonts}
\usepackage{latexsym}
\usepackage{verbatim}
\usepackage{amsmath}    % need for subequations
\usepackage{amsmath, amssymb, algorithm, algorithmic}
\usepackage[table]{xcolor}

\usepackage{multirow}
\newtheorem{assumption}{Assumption}
\usepackage [pagewise,displaymath, mathlines]{lineno}
\usepackage{amssymb}
\usepackage{epstopdf}
\usepackage{amsmath}
\usepackage{epsf}
\usepackage{graphicx}
\usepackage{mathdots}
\DeclareMathOperator*{\argmin}{arg\,min}

\newtheorem{definition}{Definition}[section]
\newtheorem{lemma}{Lemma}[section]
\newtheorem{corollary}{Corollary}[section]
\newtheorem{theorem}{Theorem}[section]

\newcommand{\beq}{\begin{equation}}
\newcommand{\eeq}{\end{equation}}
\newcommand{\beqa}{\begin{eqnarray}}
\newcommand{\eeqa}{\end{eqnarray}}
\newcommand{\beqas}{\begin{eqnarray*}}
\newcommand{\eeqas}{\end{eqnarray*}}
\newcommand{\bi}{\begin{itemize}}
\newcommand{\ei}{\end{itemize}}
\newcommand{\ba}{\begin{array}}
\newcommand{\ea}{\end{array}}

\newcommand{\nn}{\nonumber}

\def\eqnok#1{(\ref{#1})}

\def\vgap{\vspace*{.1in}}

\newcommand{\bbe}{\Bbb{E}}

\newcommand{\bbr}{\Bbb{R}}
\def\w{\omega}

\title {Approximation Methods for Bilevel Programming}

\author{Saeed Ghadimi\thanks{Department of Operations Research and Financial Engineering, Princeton University, Princeton, NJ; email: sghadimi@princeton.edu, corresponding author.}
\and
Mengdi Wang\thanks{Department of Operations Research and Financial Engineering, Princeton University, Princeton, NJ; email:
mengdiw@princeton.edu.}
}

\begin{document}
\maketitle

\begin{abstract}
In this paper, we study a class of bilevel programming problem where the inner objective function is strongly convex. More specifically, under some mile assumptions on the partial derivatives of both inner and outer objective functions, we present an approximation algorithm for solving this class of problem and provide its finite-time convergence analysis under different convexity assumption on the outer objective function. We also present an accelerated variant of this method which improves the rate of convergence under convexity assumption. Furthermore, we generalize our results under stochastic setting where only noisy information of both objective functions is available. To the best of our knowledge, this is the first time that such (stochastic) approximation algorithms with established iteration complexity (sample complexity) are provided for bilevel programming.
\end{abstract}

\section{Introduction}
We focus on the algorithms and complexity of bilevel programming problem given by
\begin{align} \label{main_prob}
&\min_{x \in X} f(x;y^*(x)) \nn \\
& \text{s.t.} \ \  y^*(x) = \argmin_{y \in \bbr^m} g(x,y)
\end{align}
where $f$ and $g$ are continuously differentiable functions and $X \subseteq \bbr^n$ is a closed convex set. In the above problem, to minimize the outer (upper-level) function $f$ with respect to (w.r.t) $x$, one needs to first minimize the inner (lower-level) function $g$ w.r.t its corresponding decision variable $y$ which depends on the outer variable $x$. This makes problem \eqnok{main_prob} intrinsically hard to solve. This class of problems was first introduced by Bracken and McGill in 70's \cite{BraMcG73}. Later, a more general form of problem \eqnok{main_prob} involving joint constraints of outer and inner variables was considered in \cite{AiyShi81}. Generally, there are a few well-known approaches for solving bilevel optimization problems. The first one is to reduce the bilevel problem to a single level problem by replacing the inner optimization problem with its optimality conditions as constraints (see e.g., \cite{HaBrGi92,ShiLuZha05}). However, this approach has two major drawbacks. If the inner problem is large-scale, then the reduced problem will have too many constraints. Also, unless $g$ has a special structure like the quadratic form, its optimality conditions typically bring nonconvexity into the feasible set of the reduced problem. Moreover, the complementarity constraints are combinatorial in nature, which lead to a mixed integer programming problem.

The second approach is to use iterative algorithms for directly solving bilevel optimization problems. Examples include descent methods using approximate descent feasible directions (see e.g., \cite{KolLas90, Falk95}), penalty methods solving penalized inner objective function (see e.g., \cite{AiyShi81, Case98}), and trust-region methods with mixed integer from of subproblems (see e.g., \cite{MarSavZh01,ColMarSav05}). Two surveys of existing works can be found in \cite{ColMarSav07,SinMalDeb17}.

To the best of our knowledge, none of the existing works provide complexity results i.e., finite-time convergence of their algorithms. In this paper, we focus on developing faster methods and provide their convergence/complexity analysis. {\color{black} To do so}, we need to make some smoothness assumptions about functions $f$, $g$, and their partial derivatives. Generally, the smoothness assumption is defined as follows.

\begin{definition}
A function $h:\bbr^{n_1} \mapsto \bbr^{n_2 \times n_3}$ is Lipschitz continuous with constant $L_h$ if
\beq \label{smoothness}
\|h(z_1) - h(z_2)\| \le L_h \|z_1-z_2\| \quad \forall z_1,z_2 \in \bbr^n,
\eeq
where $\|\cdot\|$ denotes Euclidean norm of a vector or matrix depending on the value of $n_3$.
\end{definition}

We now present two sets of assumptions on objective functions of the outer and inner problems in the form of \eqnok{main_prob}.

\begin{assumption}\label{f_assumption}
Function $f$ has the following properties.
\begin{itemize}
%\item[a)] $f(x;y^*(x))$ is continuously differentiable and has Lipschitz continuous gradient with constant $L_x$ i.e., \eqnok{smoothness} holds with $h=\nabla f(x;y^*(x))$ and $L_h = L_f>0$.

\item[a)] For any $\bar x \in X$, $\nabla_x f(\bar x;y)$ and $\nabla_y f(\bar x;y)$ are Lipschitz continuous (w.r.t $y$) with constants $L_{f_x}>0$ and $L_{f_y}>0$.

\item [b)] For any $\bar x \in X$ and $\bar y \in \bbr^m$, we have $\|\nabla_y f(\bar x;\bar y)\| \le C_{f_y}$ for some $C_{f_y}>0$.

\item[c)] For any $\bar y \in \bbr^m$, $\nabla_y f(x;\bar y)$ is Lipschitz continuous (w.r.t $x$) with constant $\bar L_{f_y}>0$.
\end{itemize}

\end{assumption}

\begin{assumption}\label{g_assumption}
Function $g$ has the following properties.
\begin{itemize}
\item [a)] For any $x \in X$ and $y \in \bbr^m$, $g$ is continuously twice differentiable in $(x,y)$.

\item [b)] For any $\bar x \in X$, $\nabla_y g(\bar x,y)$ is Lipschitz continuous (w.r.t $y$) with constant $L_g>0$.

\item [c)] For any $\bar x \in X$, $g(\bar x,y)$ is strongly convex w.r.t $y$ with parameter $\mu_g>0$ i.e., $\mu_g I \preceq \nabla_y^2 g(\bar x,y)$.

\item [d)] For any $\bar x \in X$, $\nabla_{xy}^2 g(\bar x,y)$ and $\nabla_{yy}^2 g(\bar x,y)$ are Lipschitz continuous (w.r.t $y$) with constants $L_{g_{xy}}>0$ and $L_{g_{yy}}>0$.
\item [e)] For any $\bar x \in X$ and $\bar y \in \bbr^m$, we have $\|\nabla_{xy}^2 g(\bar x,\bar y)\| \le C_{g_{xy}}$ for some constant $C_{g_{xy}}>0$.
\item [f)] For any $\bar y \in \bbr^m$, $\nabla_{xy}^2 g(x,\bar y)$ and $\nabla_{yy}^2 g(x,\bar y)$ are Lipschitz continuous (w.r.t $x$) with constants $\bar L_{g_{xy}}>0$ and $\bar L_{g_{yy}}>0$.
\end{itemize}

\end{assumption}

\vgap

We also consider the stochastic bilevel optimization problem which is a variant of problem \eqnok{main_prob} taking the form
\begin{align} \label{main_prob_st}
&\min_{x \in X} f(x;y^*(x))=\bbe[F(x;y^*(x),\xi)] \nn \\
& \text{s.t.} \ \  y^*(x) = \argmin_{y \in \bbr^m} g(x,y)= \bbe[G(x,y,\zeta)],
\end{align}
where expectations are taken with respect to independent random vectors $\xi, \zeta$ whose probability distributions are supported on $\Xi \subset \bbr^{d_1}$ and $\Omega \subset \bbr^{d_2}$, respectively. {\color{black}Since the above expectations are analytically hard to compute when $d_1,d_2 \ge 5$}, we assume that two stochastic oracles are available for each expectation such that given $(x,y)$ as an input, they output noisy partial first-order derivatives of $f$, $g$, and second-order derivatives of $g$. In particular, we make the following assumption on the output of the oracles.
\begin{assumption}\label{stochastic_assumption}
For any given $(x,y) \in \bbr^{n \times m}$, the stochastic oracles output $\nabla_x F(x;y,\xi)$, $\nabla_y F(x;y,\xi)$, $\nabla_y G(x,y,\zeta^{(1)})$, $\nabla^2_{xy} G(x,y,\zeta^{(2)})$, and $\nabla^2_{yy} G(x,y,\zeta^{(3)})$ such that
\begin{itemize}
\item [a)] $\xi,\zeta^{(1)},\zeta^{(2)}$, and $\zeta^{(3)}$ are independent random vectors.

\item [b)] $\bbe[\nabla_x F(x;y,\xi)] = \nabla_x f(x;y)$ and $\bbe[\|\nabla_x F(x;y,\xi)-\nabla_x f(x;y)\|^2] \le \sigma^2_x$.

\item [c)] $\bbe[\nabla_y F(x;y,\xi)] = \nabla_y f(x;y)$ and $\bbe[\|\nabla_y F(x;y,\xi)-\nabla_y f(x;y)\|^2] \le \sigma^2_y$.

\item [d)] $\bbe[\nabla_y G(x,y,\zeta^{(1)})] = \nabla_y g(x,y)$ and $\bbe[\|\nabla_y G(x,y,\zeta^{(1)})-\nabla_y g(x,y)\|^2] \le \sigma^2_{g_y}$.

\item [e)] $\bbe[\nabla^2_{xy} G(x,y,\zeta^{(2)})] = \nabla^2_{xy} g(x,y)$ and $\bbe[\|\nabla^2_{xy} G(x,y,\zeta^{(2)})-\nabla^2_{xy} g(x,y)\|^2] \le \sigma^2_{g_{xy}}$.

\item [f)] $\bbe[\nabla^2_{yy} G(x,y,\zeta^{(3)})] = \nabla^2_{yy} g(x,y)$ and $\bbe[\|\nabla^2_{yy} G(x,y,\zeta^{(3)})-\nabla^2_{yy} g(x,y)\|^2] \le \sigma^2_{g_{yy}}$.
\end{itemize}

\end{assumption}

\vgap

Several iterative algorithms have been also proposed for solving bilevel problems when objective functions are given in the forms of finite sum of many functions (see e.g., \cite{CouWan16,CouWan15}). {\color{black}These works consider specific examples of the objective function in \eqnok{main_prob}, so that they can be reduces and solved by variants of the classical stochastic approximation method. Asymptotic convergence of these methods are established under certain stepsize policies. However, no finite-sample convergence analysis is provided for the general case.}

{\color{black}We should point out that Assumptions~\ref{f_assumption} and \ref{g_assumption} will be made throughout the paper either for functions $f$ and $g$ (in the case of problem \eqnok{main_prob}) or for $F$ and $G$ (in the case of problem \eqnok{main_prob_st}). In the latter case, Assumption~\ref{stochastic_assumption} will be also made.

We now present some examples of bilevel optimization.}\\
$\bullet$ {\color{black}{\bf Model selection and parameter tuning}.} The first example is the selection of model parameters in learning problems. More specifically, assume that a training data set ${\cal T} = \{(x_i,y_i)\}_{i=1}^{T}$ is available to find a predictor which classifies the data points into two groups and has the least error over a validation set ${\cal V} = \{(x_j,y_j)\}_{j=1}^{V}$. {\color{black} Machine learning practitioners often formulate the empirical risk minimization problem given by }
\[
\min_\theta \left\{\frac{\lambda}{T}\sum_{i=1}^T l_i (\theta,x_i,y_i)+R(\theta)\right\},
\]
where $l_i$ is a loss function, $R(\theta)$ is a (strongly) convex regularization term, and $\lambda>0$ is a regularization parameter. The role of $\lambda$ is to balance the loss-regularization trade-off, in order to avoid overfitting the predictor over the training data set. The right choice of $\lambda$ is not known. Training the model for different values of $\lambda$ and choosing the one which has the least loss function over the validation set is a common practice. This approach requires making multiple full passes over the training and validation data sets, which is computationally expensive or even prohibited for large scale problems. Instead, {\color{black}one can reformulate the above learning problem as the following bilevel optimization problem (see e.g., \cite{KuBeHuPa06}).}
\begin{align}
&\min_{\lambda \in [0,\lambda_{\max}]} f(\lambda;\theta) = \frac{1}{V} \sum_{j=1}^V l_j (\theta(\lambda),x_j,y_j)\nn \\
& \text{s.t.} \ \  \theta(\lambda) = \argmin_{\theta \in \bbr^n} g(\lambda,\theta)= \frac{\lambda}{T} \sum_{i=1}^T l_i (\theta,x_i,y_i)+R(\theta).\nn
\end{align}
If the loss function has bounded gradient and Hessian, like the logistic function, the above $f$ and $g$ satisfy Assumptions~\ref{f_assumption} and \ref{g_assumption}.

$\bullet$ {\bf Stackelberg game model.} The second example is related to the Stackelberg model of multi-firm competition. In particular, assume that there are $N_l$ and $N_f$ firms acting as leaders and followers, respectively. Each firm wants to maximize its own profit by choosing the best amount of production. The set of leaders first act simultaneously and non-cooperatively, then the set of followers choose their own decisions. {\color{black}Denoting the production levels chosen by the leaders and followers by $x$ and $y$, respectively, the model can be formulated as
\begin{align}
&\min_{x \in \bbr^{N_l}} f(x;y) = \sum_{i=1}^{N_l} f_i(x,\bar y(x))\nn \\
& \text{s.t.} \ \  \bar y(x) = \argmin_{y \in \bbr^{N_f}} g(x,y)= \sum_{j=1}^{N_f} g_j(x,y),\nn
\end{align}
where $f_i$ and $g_j$ are the negative profit functions for the $i$-th leader and the $j$-th follower, respectively.} Some well-known models can be chosen for the profit functions like the quadratic ones which make each $g_i$ strongly convex with respect to $y_j$ and satisfy our assumption on $f$ and $g$.
%Moreover, some more constraints can be included in the model. Note that while the above problem has nonnegativity constraints for the inner problem, it can be still handled by our method by penalizing  the linear constraints in the objective function (see Section~\ref{extensions} for more details).

{\color{black}For the sake of simplicity, we define the following quantities representing complexity of finding an $\epsilon$-optimal solution (or an $\epsilon$-stationary solution) of the bilevel problem \eqnok{main_prob} i.e., $\bar x \in X$ such that $f(\bar x, y^*(\bar x)) - f^* \le \epsilon$ (or $\|\nabla f(\bar x, y^*(\bar x))\|^2 \le \epsilon$ when $f$ is possibly nonconvex and $X=\bbr^n$). In the case of problem \eqnok{main_prob_st}, we consider the optimality errors as $\bbe[f(\bar x, y^*(\bar x))]-f^*$ or $\bbe[\|\nabla f(\bar x, y^*(\bar x))\|^2]$.}

\begin{definition} \label{def_complex}
Complexity notations.\\
\begin{itemize}
\item [a)] $GC(f,\epsilon)$ ($SGC(f,\epsilon)$) denotes the total number of partial (stochastic) gradients of $f$ required to find an $\epsilon$ solution of problem \eqnok{main_prob} (\eqnok{main_prob_st}).

\item [b)] $GC(g,\epsilon)$ ($SGC(g,\epsilon)$) denotes the total number of partial (stochastic) gradients of $g$ required to find an $\epsilon$ solution of problem \eqnok{main_prob} (\eqnok{main_prob_st}).

\item [c)] $HC(g,\epsilon)$ ($SHC(g,\epsilon)$) denotes the total number of partial (stochastic) Hessians of $g$ required to find an $\epsilon$ solution of problem \eqnok{main_prob} (\eqnok{main_prob_st}).

\end{itemize}
\end{definition}

Our contributions are the following.
\begin{itemize}
\item [1)] First, we present an approximation algorithm for solving problem \eqnok{main_prob} and show that when both $f$ and $g$ are strongly convex, then this algorithm exhibits the following complexities.
\[
GC(f,\epsilon) = HC(g,\epsilon) = {\cal O} \left(\log \frac{1}{\epsilon}\right), \qquad GC(g,\epsilon)= {\cal O} \left(\log^2 \frac{1}{\epsilon}\right).
\]
To the best of our knowledge, this is the first complexity result of an iterative algorithm for solving bilevel programming. When $f$ is only convex, the above complexity bounds are changed to
\[
GC(f,\epsilon) = HC(g,\epsilon) = {\cal O} \left(\frac{1}{\epsilon}\right), \qquad GC(g,\epsilon)= {\cal O} \left(\frac{1}{\epsilon^{\tfrac54}}\right).
\]
Also, when $f$ is possibly nonconvex, this algorithm achieves the following bounds.
\[
GC(f,\epsilon) = HC(g,\epsilon) = {\cal O} \left(\frac{1}{\epsilon}\right), \qquad GC(g,\epsilon)= {\cal O} \left(\frac{1}{\epsilon^{\tfrac54}}\right).
\]

\item [2)] Second, we present an accelerated variant of the above method to further improve the aforementioned complexity bounds when $f$ is convex. In this case, our method possesses the complexity bounds of
\[
GC(f,\epsilon) = HC(g,\epsilon) = {\cal O} \left(\frac{1}{\sqrt{\epsilon}}\right), \qquad GC(g,\epsilon)= {\cal O} \left(\frac{1}{\epsilon^{\tfrac34}}\right).
\]
which are better than the aforementioned ones. This acceleration scheme does not change the complexity bounds when $f$ is strongly convex or possibly nonconvex.

\item [3)] Third, we present a stochastic variant of our algorithm for solving problem \eqnok{main_prob_st} and show that the complexity of this algorithm to find an $\epsilon$ solution of this problem, when both $f$ and $g$ are strongly convex, is bounded by
\[
SGC(f,\epsilon) = {\cal O} \left(\frac{1}{\epsilon}\right), \qquad SGC(g,\epsilon)= {\cal O} \left(\frac{1}{\epsilon^2}\right), \qquad
SHC(g,\epsilon) = {\cal O} \left(\frac{1}{\epsilon}\log \frac{1}{\epsilon}\right).
\]
When $f$ is only convex, the above bounds are changed to
\[
SGC(f,\epsilon) = {\cal O} \left(\frac{1}{\epsilon^2}\right), \qquad SGC(g,\epsilon)= {\cal O} \left(\frac{1}{\epsilon^4}\right), \qquad
SHC(g,\epsilon) = {\cal O} \left(\frac{1}{\epsilon^2}\log \frac{1}{\epsilon}\right).
\]
If $f$ is possibly nonconvex, then this algorithm exhibits the following complexity bounds.
\[
SGC(f,\epsilon) = {\cal O} \left(\frac{1}{\epsilon^2}\right), \qquad SGC(g,\epsilon)= {\cal O} \left(\frac{1}{\epsilon^3}\right), \qquad
SHC(g,\epsilon) = {\cal O} \left(\frac{1}{\epsilon^2}\log \frac{1}{\epsilon}\right).
\]
\end{itemize}

Rest of the paper is organized as follows. In Section~\ref{gbd_section}, we present an approximation method and its accelerated variant together with their convergence analysis for solving problem \eqnok{main_prob}. In Section~\ref{sab_section}, we generalize our results for solving problem the stochastic optimization problem given in \eqnok{main_prob_st}.
%We also provide some extensions to our results in Section~\ref{ext_section}.
Some concluding remarks are also presented in Section~\ref{concld_section}.

{\bf Notation.}
{\color{black}For a differentiable function $h(x,y): \bbr^{n_1 \times n_2} \mapsto \bbr$ in which $y \equiv y(x): \bbr^{n_1} \mapsto \bbr^{n_2}$, we denote its partial derivatives w.r.t $x$ and $y$ by $\nabla_x h$ and $\nabla_y h$, respectively. Moreover, $\nabla h$ is used to show the gradient of $h$ as a function of $x$.} We use $D_X$ as the diameter of the feasible set whenever it is assumed to be bounded i.e., $D_X = \max_{x,y \in X} \|x-y\|$. $f^*$ also denotes the optimal value of the outer objective function in problem \eqnok{main_prob}.
%For any $\bar x \in X$, we denote $f(\bar x) \equiv f(\bar x, y^*(\bar x))$.

%which also implies that
%\beq \label{f_smooth2}
%|f(y) - f(x) - \langle \nabla f(x), y - x \rangle | \le \frac{L}{2} \|y - x\|^{2}
%\qquad\forall x, y \in \bbr^d.
%\eeq
%We also assume that problem \eqnok{main_prob} has an optimal solution $x^*$ and $f$ satisfies
%\beq \label{very_strg_cvx}
%f(x) - f(x^*) \ge \lambda \|x-x^*\|\cdot\|x+x^*\| \ \ \forall x \in \bbr^n
%\eeq
%for some $\lambda>0$.

\section{Deterministic Approximation Methods for Bilevel Programming}\label{gbd_section}
In this section, we present algorithms for solving problem \eqnok{main_prob} when exact information of the objective functions are available. {\color{black}In Subsection~\ref{BA_sec}, we provide a deterministic approximation method and its convergence analysis for solving the bilevel programming problem. We also present an accelerated variant of this method in Subsection~\ref{ABA_sec} and show that it possesses better complexity results when the outer objective function in the bilevel programming problem is convex.

\subsection{The Bilevel Approximation Method}\label{BA_sec}
To provide an iterative algorithm for solving problem \eqnok{main_prob}, we typically need to compute the gradient of $f$ at a given point $x \in X$ which requires knowing $y^*(x)$. However, $y^*(x)$ is not available unless the inner problem has a closed-form solution which is only possible for very specific choices of function $g$. Therefore, we assume that for any $x \in X$, we have an approximation of $y^*(x)$ which is used to estimate the gradient of $f$ at $x$. More specifically, for any $\bar x \in X$ and $\bar y \in \bbr^m$, we define the gradient approximation of $f$ as
\beq
\bar \nabla f(\bar x;\bar y) := \nabla_x f(\bar x;\bar y)- M(\bar x, \bar y)\nabla_y f(\bar x;\bar y), \quad \text{where} \quad M(\bar x, \bar y) := \nabla_{xy}^2 g(\bar x,\bar y)\left[\nabla_{yy}^2 g(\bar x,\bar y)\right]^{-1}.\label{grad_f}
%M(\bar x, \bar y) &:=& \nabla_{xy}^2 g(\bar x,\bar y)\left[\nabla_{yy}^2 g(\bar x,\bar y)\right]^{-1}.\label{matrix_g}
\eeq
We discuss the above definition in more details after formally presenting our first approximation method.}

\begin{algorithm} [H]
	\caption{The Bilevel Approximation (BA) Method}
	\label{alg_BG}
	\begin{algorithmic}

\STATE Input:
$x_0 \in X$, $y_0 \in \bbr^m$ nonnegative sequences $\{\alpha_k\}_{k \ge 0}$, $\{\beta_t\}_{t \ge 0}$, and integer sequence $\{t_k\}_{k \ge 0}$.

\STATE Set $k=0$ and $\bar y_0 =y_0$.

{\bf For $k=0,1,\ldots$:}

\vgap

{\addtolength{\leftskip}{0.2in}

{\bf For $t=0,1,\ldots, t_k-1$:}

}
\vgap

{\addtolength{\leftskip}{0.4in}

\STATE Set
\beq \label{def_yt}
y_{t+1} = y_t- \beta_t \nabla_y g(x_k,y_t).
\eeq

}

{\addtolength{\leftskip}{0.2in}
{\bf End}
{\color{black}
\STATE  Set $\bar y_k=y_{t_k}$ and
\beq \label{def_xk}
x_{k+1} = \arg\min_{u \in X} \left\{\langle \bar \nabla f(x_k;\bar y_k),u \rangle + \frac{1}{2 \alpha_k}\|u-x_k\|^2 \right\},
\eeq
where $\bar \nabla f$ is defined in \eqnok{grad_f}.
}

}
{\bf End}
	\end{algorithmic}
\end{algorithm}

We now add a few remarks about the above algorithm. First, note that Algorithm~\ref{alg_BG} consists of two iterative loops. The outer loop indexed by $k$, counts the number of inexact projected gradient performed on function $f$ in \eqnok{main_prob} over the feasible set of the outer variable $x \in X$. The inner one indexed by $t$, shows steps of the gradient method with respect to the inner variable and function $y$ and $g$, respectively. Second, the number of iterations of the inner loop plays a key role in the convergence analysis of the above algorithm and it should be specified at each iteration of the outer loop. In particular, the larger this number is, the more accurate one solves the inner minimization problem in \eqnok{main_prob}. On the other hand, accuracy of the output solution from the inner loop affects the total complexity of Algorithm~\ref{alg_BG}. We will discuss this issue later. Finally, note that the error in gradient approximation of $f$ defined in \eqnok{grad_f} should be controlled appropriately to enable us providing convergence analysis of Algorithm~\ref{alg_BG}. In the next couple of technical results, we show that how this approximation relates to the true gradient of $f$ and its error can be controlled by the solution of the inner loop in the above algorithm.

\begin{lemma} Suppose that Assumptions~\ref{g_assumption}.a) and .b) hold.
\begin{itemize}
\item [a)] For any $\bar x \in X$, $y^*(\bar x)$ is unique and differentiable and we have
\beq\label{grad_ystar}
\nabla y^*(\bar x) = - M(\bar x, y^*(\bar x))^\top,
\eeq
where matrix $M$ is defined in \eqnok{grad_f}.

\item [b)] For any $\bar x \in X$, gradient of $f$ as a function of $x$, is given by
{\color{black}
\beq\label{grad_f2}
\nabla f(\bar x; y^*(\bar x)) = \nabla_x f(\bar x; y^*(\bar x))- M(\bar x, y^*(\bar x)) \nabla_y f(\bar x; y^*(\bar x)).
\eeq
}

\end{itemize}
\end{lemma}

\begin{proof}
{\color{black}The above results have been well-known from properties of implicit functions. For the sake of completeness, we provide briefly their proofs. Due to the definition of $y^*(x)$ in \eqnok{main_prob}, we have $\nabla_y g(\bar x,y^*(\bar x))=0$ due to the optimality condition of the inner problem. Then, by taking derivative on both sides, using the chain rule, and the implicit function theorem, we obtain
\[
\nabla_{yx}^2 g(\bar x,y^*(\bar x)) +  \nabla_{yy}^2 g(\bar x,y^*(\bar x)) \nabla y^*(\bar x)=0,
\]
which under Assumption~\ref{g_assumption}.c) and in the view of \eqnok{grad_f}, imply  \eqnok{grad_ystar}.} Part b) then follows immediately due to the chain rule.
\end{proof}

\vgap

Our next result measures the error in estimation of gradient of $f$.

\begin{lemma} \label{grad_f_error}
The following statements hold.
\begin{itemize}
\item [a)] Suppose that $\bar x \in X$ and $\bar y \in \bbr^m$ are given and Assumptions~\ref{f_assumption} and ~\ref{g_assumption} hold. Then, we have
\beq \label{def_grad_error}
\|\bar \nabla f(\bar x; \bar y) - \nabla f(\bar x; y^*(\bar x))\| \le C \|y^*(\bar x)-\bar y\|,
\eeq
where $C=L_{f_x}+\frac{L_{f_y}C_{g_{xy}}}{\mu_g} + C_{f_y} \left[\frac{L_{g_{xy}}}{\mu_g} +\frac{L_{g_{yy}}C_{g_{xy}}}{\mu_g^2}\right]$.

\item [b)] Under Assumptions~\ref{g_assumption}.c) and e), $y^*(x)$ is Lipschitz continuous in $x$ with constant $\frac{C_{g_{xy}}}{\mu_g}$.

\item [c)] Under Assumptions~\ref{f_assumption} and ~\ref{g_assumption}, $\nabla f$ is Lipschitz continuous in $x$ with constant $L_f$ i.e., for any given $\bar x_1, \bar x_2 \in X$, we have
\beq \label{def_f_smooth}
\|\nabla f(\bar x_2; y^*(\bar x_2)) - \nabla f(\bar x_1; y^*(\bar x_1))\| \le L_f \|\bar x_2 - \bar x_1\|,
\eeq
where $L_f =\frac{(\bar L_{f_y}+C) \cdot C_{g_{xy}}}{\mu_g}+L_{f_x}+C_{f_y}\left[\frac{\bar L_{g_{xy}}C_{f_y}}{\mu_g} +\frac{\bar L_{g_{yy}}C_{g_{xy}}}{\mu_g^2}\right]$.
\end{itemize}
\end{lemma}

\begin{proof}
First, denoting
\beqa
\Delta_k &=& \nabla f(\bar x; \bar y) - \nabla f(\bar x; y^*(\bar x)),\ \ \ \ \Delta_k^1 = \nabla_x f(\bar x; \bar y) - \nabla_x f(\bar x; y^*(\bar x)), \nn \\
\Delta_k^2 &=& M(\bar x, \bar y)\nabla_y f(\bar x;\bar y) - M(\bar x, y^*(\bar x))\nabla_y f(\bar x; y^*(\bar x))\nn \\
%\Delta_k^2 &=& \nabla_{xy}^2 g(\bar x,\bar y)\left[\nabla_{yy}^2 g(\bar x,\bar y)\right]^{-1}\nabla_y f(\bar x;\bar y)-\nabla_{xy}^2 g(\bar x,y^*(\bar x))\left[\nabla_{yy}^2 g(\bar x,y^*(\bar x))\right]^{-1}\nabla_y f(\bar x; y^*(\bar x)), \nn \\
\Delta_k^3 &=& M(\bar x, \bar y)\left\{\nabla_y f(\bar x;\bar y)-\nabla f(\bar x; y^*(\bar x))\right\}, \nn \\
\Delta_k^4 &=& \left\{M(\bar x, \bar y)-M(\bar x, y^*(\bar x)) \right\}\nabla_y f(\bar x; y^*(\bar x)), \nn \\
\Delta_k^5 &=& \left\{\nabla_{xy}^2 g(\bar x;\bar y)-\nabla_{xy}^2 g(\bar x,y^*(\bar x))\right\}\left[\nabla_{yy}^2 g(\bar x,\bar y)\right]^{-1}, \nn \\
\Delta_k^6 &=& \nabla_{xy}^2 g(\bar x,y^*(\bar x))
\left\{\left[\nabla_{yy}^2 g(\bar x;\bar y)\right]^{-1}-\left[\nabla_{yy}^2 g(\bar x,y^*(\bar x))\right]^{-1}\right\}, \nn
\eeqa
and in the view of \eqnok{grad_f} and \eqnok{grad_f2}, we obtain
\beq
\Delta_k = \Delta_k^1 - \Delta_k^2 = \Delta_k^1 - \Delta_k^3 -\Delta_k^4 = \Delta_k^1 - \Delta_k^3 -\left(\Delta_k^5+\Delta_k^6 \right) \nabla_y f(\bar x, y^*(\bar x)).\label{proof_error1}
\eeq
Now, under Assumptions~\ref{f_assumption} and ~\ref{g_assumption}, we have
\beqa
\|\Delta_k^1\| &\le& L_{f_x}\|y^*(\bar x)-\bar y\|, \ \
\|\left[\nabla_{yy}^2 g\right]^{-1}\| \le \frac{1}{\mu_g}, \ \
\|\Delta_k^3\| \le \frac{L_{f_y}C_{g_{xy}}}{\mu_g} \|y^*(\bar x)-\bar y\|, \nn \\
\|\nabla_y f\| &\le& C_{f_y}, \ \ \|\Delta_k^5\| \le \frac{L_{g_{xy}}}{\mu_g} \|y^*(\bar x)-\bar y\|, \ \
\|\Delta_k^6\| \le \frac{L_{g_{yy}}C_{g_{xy}}}{\mu_g^2} \|y^*(\bar x)-\bar y\|,\label{proof_error2}
\eeqa
where the last inequality follows from the fact that
\[
\|H_2^{-1} - H_1^{-1}\| = \|H_1^{-1}\left(H_1 - H_2 \right)H_2^{-1}\| \le \|H_1^{-1}\|\|H_2^{-1}\|\|H_1 - H_2\|
\]
for any invertible matrices $H_1$ and $H_2$. Combining \eqnok{proof_error1} and \eqnok{proof_error2} with Cauchy-Schwarz inequality, we obtain \eqnok{def_grad_error}.

Second, noting \eqnok{grad_ystar} and \eqnok{proof_error2}, we have
\[
\|\nabla y^*(\bar x)\| = \|M(\bar x, y^*(\bar x))\| \le \frac{C_{g_{xy}}}{\mu_g},
\]
which clearly implies part b).

Third, noting \eqnok{grad_f}, we have
\beqa
\|\nabla f(\bar x_2; y^*(\bar x_2)) - \nabla f(\bar x_1; y^*(\bar x_1))\| &\le& \|\nabla f(\bar x_2; y^*(\bar x_2)) - \bar \nabla f(\bar x_2; y^*(\bar x_1))\| \nn \\
&+& \|\bar \nabla f(\bar x_2; y^*(\bar x_1)) - \nabla f(\bar x_1; y^*(\bar x_1))\|.\nn
\eeqa
Then, \eqnok{def_f_smooth} follows similarly to part a) by noting part b).
\end{proof}

\vgap

{\color{black}It should be mentioned that results of Lemma~\ref{grad_f_error}.b) and .c) have been also shown in \cite{CouWan16} under slightly different assumptions for the special case of problem \eqnok{main_prob} where $x$ does not explicitly appear in the definition of $f$.}
Next result establishes convergence of the inner loop in Algorithm~\ref{alg_BG}, which essentially follows from convergence analysis of the classical the gradient method.
\begin{lemma}\label{lemma_gd}
Let $\{y_t\}_{t=0}^{t_k}$ be the sequence generated at the $k$-th iteration of Algorithm~\ref{alg_BG}, Assumptions~\ref{g_assumption}.b), and ~\ref{g_assumption}.c) hold. If $\beta_t = 2/(\mu_g+L_g) \ \ t \ge 0$, we have
\beq\label{y_cnvrg}
\|y_{t_k} - y^*(x_k) \| \le \left(\frac{Q_g-1}{Q_g+1}\right)^{t_k} \|y_0 - y^*(x_k) \|,
\eeq
where $Q_g = L_g/\mu_g$ denote the condition number of the inner function $g$.
\end{lemma}

\begin{proof}
Note that $y^*(x_k)$ is the optimal solution of inner problem in \eqnok{main_prob} when $x=x_k$. Then, \eqnok{y_cnvrg} follows from the standard proofs for the gradient descent method when applied to smooth strongly convex problems (see e.g., \cite{Nest04}).
\end{proof}

\vgap

The above results show that there is a trade-off between inexactness of the gradient estimation of $f$ in \eqnok{grad_f} and accuracy of the solution obtained by the inner loop of Algorithm~\ref{alg_BG}. We are now ready to present the main convergence results of this algorithm.

\begin{theorem}[Convergence results for the BA algorithm]\label{theom_main_bg}
Suppose that $\{\bar y_k, x_k \}_{k \ge 0}$ is generated by Algorithm~\ref{alg_BG}, Assumptions~\ref{f_assumption} and ~\ref{g_assumption} hold, and stepsizes are chosen such that
\beq \label{alpha_beta}
\beta_t = \frac{2}{L_g+\mu_g}  \ \ \forall t \ge 0, \quad \alpha_k \le \frac{1}{L_f} \ \ \forall k \ge 0.
\eeq
\begin{itemize}
\item [a)] If $f$ is strongly convex with parameter $\mu_f>0$, we have, for any $N \ge 1$,
\beq\label{bg_strng_cvx}
f(x_N;y^*(x_N))-f^* \le \Gamma_N \left[f(x_0;y^*(x_0))-f^*+\frac{C^2}{2} \sum_{k=0}^{N-1}\frac{\alpha_k A_k}{\Gamma_{k+1}} \right],
\eeq
where
\beq \label{def_Ak}
A_k = \|y_0 - y^*(x_k) \|^2\left(\frac{Q_g-1}{Q_g+1}\right)^{2t_k},
\eeq
\beq \label{def_Gamma}
\Gamma_1 := \left\{
\begin{array}{ll}
 1, & \gamma_0 = 1,\\
1 - \gamma_0, & \gamma_0<1,
\end{array} \right. \ \
\Gamma_{k} := \Gamma_1 \prod_{i=1}^{k-1} (1 - \gamma_i)  \ \ \forall k \ge 2,
\eeq
and
\beq\label{def_gama}
0<\gamma_k \le \alpha_k \mu_f \ \ \forall k \ge 0.
%\gamma_k = \min\left(\alpha_k \mu_f,\frac{2}{Q_g+1}\right) \ \ \forall k \ge 0.
\eeq
\item [b)] If $f$ is convex and $X$ is bounded, we have
\beq \label{bg_cvx}
f(\bar x_N;y^*(\bar x_N))-f^* \le \frac{1}{N} \sum_{k=0}^{N-1}\left(\frac{1}{2 \alpha_k}\left[\|x^*-x_k\|^2-\|x^*-x_{k+1}\|^2\right]+C D_X \sqrt{A_k} \right),
\eeq
where
\beq\label{def_xave}
\bar x_N = \frac{\sum_{k=1}^{N} x_k}{N}.
\eeq

\item [c)] If $f$ is possibly nonconvex, $X=\bbr^n$ (for simplicity), and stepsizes are chosen such that $\alpha_k < 1/(2L_f)$, we have
\beq\label{bg_nocvx}
\sum_{k=0}^{N-1} \frac{\alpha_k}{2} (1-2L_f \alpha_k) \|\nabla f(x_k;y^*(x_k))\|^2 \le f(x_0;y^*(x_0))-f^*+\frac{C^2}{2} \sum_{k=0}^{N-1} \left[\alpha_k(1+2L_f\alpha_k)A_k \right].
\eeq
\end{itemize}
\end{theorem}

\begin{proof}
We first show part a). Noting that subproblem \eqnok{def_xk} is strongly convex, we have
\[
\langle \bar \nabla f(x_k;\bar y_k),x_{k+1}-u \rangle \le \frac{1}{2 \alpha_k}\left[\|u-x_k\|^2-\|u-x_{k+1}\|^2-\|x_{k+1}-x_k\|^2 \right] \ \ \forall u \in X.
\]
Moreover, noting the smoothness of $f$ due to Lemma~\ref{grad_f_error}.c), we have
\[
f(x_{k+1};y^*(x_{k+1})) \le f(x_k;y^*(x_k)) + \langle \nabla f(x_k;y^*(x_k)),x_{k+1}-x_k \rangle +\frac{L_f}{2}\|x_{k+1}-x_k\|^2.
\]
Adding the above inequalities, denoting $\Delta_k \equiv \bar \nabla f(x_k;\bar y_k)-\nabla f(x_k;y^*(x_k))$, and re-arranging the terms,  we obtain
\begin{align}
f(x_{k+1}; y^*(x_{k+1})) \le f(x_k;y^*(x_k)) &+ \langle \nabla f(x_k;y^*(x_k)),u-x_k \rangle +\frac{1}{2 \alpha_k}\left[\|u-x_k\|^2-\|u-x_{k+1}\|^2\right] \nn \\
&- \frac{(1-L_f \alpha_k)}{2 \alpha_k} \|x_{k+1}-x_k\|^2 +  \langle \Delta_k ,u-x_{k+1} \rangle \quad \forall u \in X, \label{BG_proof1}
\end{align}
which together with the choice of $\alpha_k$ in \eqnok{alpha_beta} and the fact that
\beq \label{delta_CS_ineq}
\langle \Delta_k ,u-x_{k+1} \rangle \le \|\Delta_k\| \cdot \|u-x_{k+1}\| \le \frac{a}{2}\|\Delta_k\|^2 + \frac{1}{2a}\|u-x_{k+1}\|^2,
\eeq
(with $a=\alpha_k$) imply that
\[
f(x_{k+1};y^*(x_{k+1})) \le f(x_k;y^*(x_k)) + \langle \nabla f(x_k;y^*(x_k)),u-x_k \rangle +\frac{1}{2 \alpha_k}\|u-x_k\|^2+\frac{\alpha_k}{2}\|\Delta_k\|^2 \ \ \forall u \in X.
\]
Letting $u = \theta_k x^*+(1-\theta_k)x_k$ for some $\theta_k \in [0,1]$ in the above inequality, noting strong convexity of $f$, the choice of $\beta_k$ in \eqnok{alpha_beta}, \eqnok{y_cnvrg}, and \eqnok{def_grad_error} we have
\beqa
f(x_{k+1};y^*(x_{k+1})) &\le& (1-\theta_k)f(x_k;y^*(x_k)) + \theta_k \left[f(x_k;y^*(x_k))+\langle \nabla f(x_k;y^*(x_k)),x^*-x_k \rangle \right]\nn \\
&&\qquad \qquad \qquad \qquad \qquad \qquad \qquad \quad \quad \qquad+  \frac{\theta_k^2}{2 \alpha_k}\|x^*-x_k\|^2+\frac{\alpha_k}{2}\|\Delta_k\|^2 \label{BG_proof2} \\
&\le& (1-\theta_k)f(x_k;y^*(x_k)) + \theta_k f(x^*;y^*(x^*)) -\frac{\theta_k \mu_f }{2}\left(1- \frac{\theta_k}{\mu_f \alpha_k}\right) \|x^*-x_k\|^2 \nn \\
&& \qquad \qquad \qquad \qquad \qquad \qquad \qquad \qquad \quad+\frac{\alpha_k \|y_0 - y^*(x_k) \|^2}{2}\left(\frac{Q_g-1}{Q_g+1}\right)^{2t_k}  \nn \\
&=& (1-\alpha_k \mu_f)f(x_k;y^*(x_k)) + \alpha_k \mu_f f(x^*;y^*(x^*))+\frac{C^2 \alpha_k \|y_0 - y^*(x_k) \|^2}{2}\left(\frac{Q_g-1}{Q_g+1}\right)^{2t_k}, \nn
\eeqa
where the last equality follows from choosing $\theta_k =\alpha_k \mu_f$ which is less than $1$ due to \eqnok{alpha_beta}. Subtracting $f(x^*;y^*(x^*))$ from both sides of the above inequality, noting \eqnok{def_Ak}, and \eqnok{def_gama}, we have
\[
f(x_{k+1};y^*(x_{k+1}))-f(x^*;y^*(x^*)) \le (1-\gamma_k) [f(x_k;y^*(x_k))-f(x^*;y^*(x^*))]+\frac{C^2 \alpha_k A_k}{2}.
\]
dividing both sides by $\Gamma_{k+1}$, summing them up by noting \eqnok{def_Gamma}, we obtain \eqnok{bg_strng_cvx}.

We now show part b). Setting $u=x^*$ in \eqnok{BG_proof1}, noting convexity of $f$, and boundedness of $X$, the first inequality in \eqnok{delta_CS_ineq}, and after re-arranging the terms, we obtain
\[
f(x_{k+1};y^*(x_{k+1}))-f(x^*;y^*(x^*)) \le \frac{1}{2 \alpha_k}\left[\|x^*-x_k\|^2-\|x^*-x_{k+1}\|^2\right]+D_X \|\Delta_k\|.
\]
Summing up both sides of the above inequality and then dividing them by $N$, we obtain \eqnok{bg_cvx} in the view of \eqnok{def_xave} due to the convexity of $f$.

Finally, we show part c). If $f$ is possibly nonconvex and $X=\bbr^n$, then by \eqnok{def_xk}, we have $x_{k+1} = x_k - \alpha_k \bar \nabla f(x_k;\bar y_k)$ which together with the choice of $u=x_{k+1}$ in \eqnok{BG_proof1}, imply that
\beqa
f(x_{k+1};y^*(x_{k+1})) &\le& f(x_k;y^*(x_k)) -\alpha_k \|\nabla f(x_k;y^*(x_k))\|^2 - \alpha_k \langle \nabla f(x_k;y^*(x_k)), \Delta_k \rangle  \nn \\
&& \qquad \qquad \qquad \qquad \qquad \qquad \qquad \qquad+\frac{L_f \alpha_k^2}{2}\|\nabla f(x_k;y^*(x_k))+\Delta_k\|^2 \nn \\
&\le& f(x_k;y^*(x_k)) -\frac{\alpha_k}{2}(1-2L_f \alpha_k) \|\nabla f(x_k;y^*(x_k))\|^2 +\frac{\alpha_k}{2}(1+2L_f \alpha_k) \|\Delta_k\|^2.\nn \\\label{nocvx_p1}
\eeqa
Summing up both sides of the above inequality, re-arranging the terms, noting that $\alpha_k < 1/(2L_f)$, \eqnok{y_cnvrg}, \eqnok{def_grad_error}, and \eqnok{def_Ak}, we obtain \eqnok{bg_nocvx}.
\end{proof}

\vgap

{\color{black}In the next result, we specialize rates of convergence of Algorithm~\ref{alg_BG} when applied to problem \eqnok{main_prob} under different convexity assumptions on $f$.}
\begin{corollary} \label{lemma_main_bg}
Suppose that $\{\bar y_k, x_k \}_{k \ge 0}$ is generated by Algorithm~\ref{alg_BG}, Assumptions~\ref{f_assumption} and ~\ref{g_assumption} hold, $\beta_k$ is set to \eqnok{alpha_beta}, and
\beq \label{alpha_beta1}
\alpha_k = \frac{1}{3L_f} \ \ \forall k \ge 0.
\eeq
\begin{itemize}
\item [a)] If $f$ is strongly convex with parameter $\mu_f>0$, and $t_k =k+1$, we have, for any $N \ge 1$,
\beq\label{bg_strng_cvx1}
f(x_N;y^*(x_N))-f^* \le (1-\gamma)^N\left[f(x_0;y^*(x_0))-f^*+\frac{(Q_g-1) M^2 C^2}{6 L_f} \right],
\eeq
where
\beq \label{def_param}
\gamma_k = \gamma = \min \left(\frac{\mu_f}{3L_f},\frac{2}{Q_g+1}\right) \quad \forall k \ge 0, \qquad M = \max_{x \in X}\|y_0 - y^*(x) \|.
\eeq

\item [b)] If $f$ is convex, $X$ is bounded, and $t_k = \lceil \sqrt[4]{k+1}\rceil$, we have
\beq \label{bg_cvx1}
f(x_N;y^*(x_N))-f^* \le \frac{18 L_f D_X^2}{N} + \frac{(Q_g-1)^2(Q_g+1)^6 C^2 M^2}{75 L_f N}.
\eeq

\item [c)] If $f$ is possibly nonconvex, $X=\bbr^n$ (for simplicity), and $t_k = \lceil \frac{\sqrt[4]{k+1}}{2} \rceil$, we have
\beq\label{bg_nocvx1}
\bbe\left[\|\nabla f(x_R;y^*(x_R))\|^2\right] \le \frac{18L_f[f(x_0;y^*(x_0))-f^*]+5(Q_g-1)(Q_g+1)^3 C^2 M^2}{N},
\eeq
where the expectation is taken with respect to the integer random variable $R$ uniformly distributed over $\{0,1,\ldots, N-1\}$.
\end{itemize}
\end{corollary}

\begin{proof}
First, note that choices of $\alpha_k$, $\gamma_k$ in \eqnok{alpha_beta1}, \eqnok{def_param} satisfy \eqnok{alpha_beta}, \eqnok{def_gama}, and together with \eqnok{def_Gamma} and choice of $t_k =2k+1$, we have
\begin{align}
&\Gamma_k = (1-\gamma)^N \ge \left(\frac{Q_g-1}{Q_g+1}\right)^N, \nn \\
&\sum_{k=0}^{N-1}\frac{\alpha_k A_k}{\Gamma_{k+1}} \le \frac{M^2}{3 L_f} \sum_{k=0}^{N-1}  \left(\frac{Q_g-1}{Q_g+1}\right)^{2t_k-k-1}
= \frac{M^2}{3 L_f} \sum_{k=0}^{N-1}  \left(\frac{Q_g-1}{Q_g+1}\right)^{k+1} \le \frac{(Q_g-1) M^2}{6 L_f},\label{gamma_ineq}
\end{align}
which together with \eqnok{bg_strng_cvx}, imply \eqnok{bg_strng_cvx1}. Second, observe that with the choice of $\alpha_k$ in \eqnok{alpha_beta1} and for any $\rho \in (0,1)$, we have
\begin{align}
&\sum_{k=0}^{N-1}\frac{1}{2 \alpha_k}\left[\|x^*-x_k\|^2-\|x^*-x_{k+1}\|^2\right] \le \frac{3L_f}{2} \|x^*-x_0\|^2, \nn \\
&\sum_{k=0}^{N-1}  \rho^{\sqrt[4]{k+1}} \le \sum_{k=1}^{\lfloor \sqrt[4]{N} \rfloor} [(k+1)^4-k^4]\rho^{k} \le \frac{15 \rho}{(1-\rho)^4},
%& \Gamma_k = \frac{2}{k(k+1)}, \qquad \sum_{k=0}^{N-1} \frac{\gamma_k^2}{\Gamma_{k+1}} \le 2 N,
\label{power_sum}
\end{align}
where the equality follows from \eqnok{def_Gamma}. Combining the above observations with the choice of $t_k = \lceil \sqrt[4]{k+1} \rceil$ in \eqnok{bg_cvx}, we obtain \eqnok{bg_cvx}. Third, noting \eqnok{alpha_beta1} and the choices of $t_k = \lceil \sqrt[4]{k+1}/2 \rceil$ , we have
\begin{align}
&\sum_{k=0}^{N-1} \frac{\alpha_k}{2} (1-2L_f \alpha_k)  = \frac{N}{18 L_f}, \nn \\
&\sum_{k=0}^{N-1} \left[\alpha_k(1+2L_f\alpha_k)A_k \right] \le \frac{5M^2}{9 L_f} \sum_{k=0}^{N-1}
\left(\frac{Q_g-1}{Q_g+1}\right)^{\sqrt[4]{k+1}} \le \frac{25(Q_g-1)(Q_g+1)^3 M^2}{96 L_f},\nn
\end{align}
where the last inequality follows similarly to \eqnok{power_sum}. Combining the above relations with \eqnok{bg_nocvx1}, and in the view of
\[
\bbe\left[\|\nabla f(x_R;y^*(x_R))\|^2\right] = \frac{\sum_{k=0}^{N-1} \|\nabla f(x_k;y^*(x_k))\|^2}{N},
\]
we obtain \eqnok{bg_nocvx}.
\end{proof}

\vgap
we make a few remarks bout the above results in Corollary~\ref{lemma_main_bg}. First, observe that the total number of iterations performed by the inner loop till the $k$-th iteration of the outer loop is $\sum_{i=0}^{k}t_i$, which together with the choice of $t_k=k+1$, \eqnok{bg_strng_cvx1}, and \eqnok{def_param} imply that the iteration complexities of Algorithm~\ref{alg_BG} to find an $\epsilon$ solution of problem \eqnok{main_prob}, in the view of Definition~\ref{def_complex}, are bounded by
\beq \label{complex_bnd_strcvx}
GC(f,\epsilon) = HC(g,\epsilon) = \max\left\{\frac{L_f}{\mu_f},\frac{L_g}{\mu_g}\right\} {\cal O} \left(\log \frac{1}{\epsilon}\right), \qquad GC(g,\epsilon)= GC(f,\epsilon)^2,
\eeq
when $f$ is strongly convex. Note that $GC(f,\epsilon)$ is in the same order of the optimal complexity bound for smooth strongly convex optimization. Second, similarly \eqnok{bg_cvx1} implies that the above bounds are change to
\beq \label{complex_bnd_cvx}
GC(f,\epsilon) = HC(g,\epsilon) = \max\left\{L_f D_X^2,\frac{Q_g^8 C^2 M^2}{L_f}\right\} {\cal O} \left(\frac{1}{\epsilon}\right), \qquad GC(g,\epsilon)= GC(f,\epsilon)^{\tfrac{5}{4}},
\eeq
when $f$ is only convex.  Note that is $GC(f,\epsilon)$ similar to the complexity bound of the gradient descent applied to convex programming. Third, when $f$ is possibly nonconvex, to find an $\epsilon$ solution in terms of where the expectation is taken with respect to the uniform distribution, Algorithm~\ref{alg_BG} exhibits complexity bounds in the order of
\beq \label{complex_bnd_nocvx}
GC(f,\epsilon) = HC(g,\epsilon) = \max\left\{L_f D_X^2,\frac{Q_g^4 C^2 M^2}{L_f}\right\} {\cal O} \left(\frac{1}{\epsilon}\right), \qquad GC(g,\epsilon)= GC(f,\epsilon)^{\tfrac{5}{4}},
\eeq
Fourth, note that the aforementioned complexity bounds are obtained through a unified analysis in the sense that Algorithm~\ref{alg_BG} is using one stepsize policy and is implemented regardless of the convexity of $f$. However, its complexity behaviour clearly depends on the convexity of $f$. To the best of our knowledge, this is the first time that iteration complexity bounds are provided for iterative algorithms when applied to bilevel optimization problems.

{\color{black}Finally, note that the first complexity bound in \eqnok{complex_bnd_cvx} (in terms of gradient computation of $f$) does not match the lower bound of ${\cal O}(1/\sqrt{\epsilon})$ for convex programming \cite{nemyud:83}. This motivates us to use an acceleration scheme, similar to the classic convex programming, to improve this complexity bound.

\subsection{The Accelerated Bilevel Approximation Method}\label{ABA_sec}

In this subsection, we first present an accelerated variant of Algorithm~\ref{alg_BG} and then present its convergence analysis.}

%We now present an accelerated variant of Algorithm~\ref{alg_BG} which improve some of the aforementioned complexity bounds
\begin{algorithm} [H]
	\caption{The Accelerated Bilevel Approximation (ABA) Method}
	\label{alg_abg}
	\begin{algorithmic}

\STATE Input:
$x_0 \in X$, $y_0 \in \bbr^m$ nonnegative sequences $\{\theta_k\}_{k \ge 0} \in (0,1]$, $\{\alpha_k\}_{k \ge 0}$, $\{\lambda_k\}_{k \ge 0}$, $\{\beta_t\}_{t \ge 0}$, and integer sequence $\{t_k\}_{k \ge 0}$.

\STATE Set $k=0$, $x^{ag}_0=x_0$, and $\bar y_0 =y_0$.

{\bf For $k=0,1,\ldots$:}

\vgap

{\addtolength{\leftskip}{0.2in}

\STATE Set
\beq \label{def_x_md}
\eta_k = \frac{\theta_k(\mu_f+\lambda_k) - \theta_k^2 \mu_f}{\mu_f+\lambda_k - \theta_k^2 \mu_f} \qquad \text{and} \qquad x^{md}_k = \eta_k x_k + (1-\eta_k) x^{ag}_k.
\eeq
{\bf For $t=0,1,\ldots, t_k-1$:}

}
\vgap

{\addtolength{\leftskip}{0.4in}

\STATE Set
\beq \label{def_yt}
y_{t+1} = y_t- \beta_t \nabla_y g(x^{md}_k,y_t).
\eeq

}

{\addtolength{\leftskip}{0.2in}
{\bf End}

\STATE  Set $\bar y_k=y_{t_k}$ and compute $\bar \nabla f(x^{md}_k;\bar y_k)$ according to \eqnok{grad_f} and set

\beqa
x_{k+1} &=& \arg\min_{u \in X} \left\{\langle \bar \nabla f(x^{md}_k;\bar y_k),u \rangle + \frac{\mu_f}{4}\|u-x^{md}_k\|^2 +\frac{(1-\theta_k)\mu_f+\lambda_k}{4 \theta_k}\|u-x_k\|^2 \right\},\label{def_xk_ac}\\
x^{ag}_{k+1} &=& \arg\min_{u \in X} \left\{\langle \bar \nabla f(x^{md}_k;\bar y_k),u \rangle + \frac{1}{2 \alpha_k}\|u-x^{md}_k\|^2 \right\}.\label{def_xag}
\eeqa

}
{\bf End}
	\end{algorithmic}
\end{algorithm}

\vgap

Note that if $\bar \nabla f(x^{md}_k;\bar y_k) = \nabla f(x^{md}_k;y^*(x^{md}_k))$, then \eqnok{def_x_md}, \eqnok{def_xk_ac}, and \eqnok{def_xag} form a variant of the accelerated gradient method proposed by Nesterov~\cite{Nest04}. Moreover, the above algorithm similarly to Algorithm~\ref{alg_BG} has two nested loops where the inner ones are the same. Indeed, the acceleration scheme is implemented through the outer loop. Since the inner function $g$ in problem \eqnok{main_prob} is always assumed to be strongly convex, incorporating acceleration scheme into the inner loop of Algorithm~\ref{alg_abg} only improves the complexity bounds in terms of the dependence on the condition number of $g$. Hence, for sake of simplicity, we still apply the gradient method in the inner loop of Algorithm~\ref{alg_abg}.
%It is worth noting that the overall iteration complexity of the inner problem is also reduced as an
Below, we present the main convergence properties of this algorithm.

\begin{theorem}[Convergence results for the ABA algorithm]\label{theom_main_abg}
Suppose that $\{\bar y_k, x_k, x^{md}_k, x^{ag}_k \}_{k \ge 0}$ is generated by Algorithm~\ref{alg_abg}, Assumptions~\ref{f_assumption} and ~\ref{g_assumption} hold, stepsizes are chosen such that \eqnok{alpha_beta} holds, and
\beq \label{alpha_lambda}
\theta_k^2 \le \frac{\alpha_k(\mu_f+\lambda_k)}{4} \ \ \forall k \ge 0.
\eeq
\begin{itemize}
\item [a)] If $f$ is strongly convex with parameter $\mu_f>0$ and
\beq \label{alpha_lambda1}
\frac{\lambda_0}{\Gamma_1} = \frac{\lambda_1}{\Gamma_2} = \ldots,
\eeq
we have, for any $N \ge 1$,
\beq
f(x^{ag}_N;y^*(x^{ag}_N))-f^* \le \Gamma_N \left[f(x_0;y^*(x_0))-f^* + \frac{\mu_f+\lambda_0 \Gamma^{-1}_1}{4}\|x^*-x_0\|^2 +\frac{C^2}{2\mu_f} \sum_{k=0}^{N-1}\frac{(6\theta_k + \alpha_k \mu_f) A_k}{\Gamma_{k+1}} \right],\label{abg_strng_cvx}
\eeq
where $A_k$ is given by \eqnok{def_Ak} and $\Gamma_k$ is defined in \eqnok{def_Gamma} with
\beq\label{def_gama2}
0<\gamma_k \le \frac{\sqrt{\alpha_k \mu_f}}{2} \qquad \forall k \ge 0.
\eeq

\item [b)] If $f$ is convex, $X$ is bounded, and \eqnok{alpha_lambda1} holds, we have
\begin{align}
f(x^{ag}_N;y^*(x^{ag}_N))-f^* &\le \Gamma_N \left[\frac{(1-\gamma_0)[f(x_0;y^*(x_0))-f^*]}{\Gamma_1}+\frac{\lambda_{0}}{\Gamma_1} \|x^*-x_0\|^2 \right. \nn \\
& \qquad \qquad \qquad \qquad \qquad  + \left. \sum_{k=0}^{N-1} \frac{C}{\Gamma_{k+1}} \left(\theta_k D_X \sqrt{A_k}+\frac{C \alpha_k A_k}{2} \right)\right],\label{abg_cvx}
\end{align}
where $\Gamma_k$ is defined for $\gamma_k = \theta_k$.
\end{itemize}
\end{theorem}

\vgap

\begin{proof}
We first show part a). Noting strong convexity of subproblems \eqnok{def_xk_ac} and \eqnok{def_xag}, we have
\begin{align}
&\langle \bar \nabla f(x^{md}_k;\bar y_k),x_{k+1}-x \rangle \le \frac{\mu_f}{4}\left[\|x-x^{md}_k\|^2-\|x-x_{k+1}\|^2-\|x_{k+1}-x^{md}_k\|^2 \right] \nn \\
& \qquad \qquad \qquad \qquad+ \frac{(1-\theta_k)\mu_f+\lambda_k}{4 \theta_k} \left[\|x-x_k\|^2-\|x-x_{k+1}\|^2-\|x_{k+1}-x_k\|^2 \right]
\ \ \forall x \in X, \label{abg_proof1} \\
&\langle \bar \nabla f(x^{md}_k;\bar y_k),x^{ag}_{k+1}-u \rangle \le \frac{1}{2 \alpha_k}\left[\|u-x^{md}_k\|^2-\|u-x^{ag}_{k+1}\|^2-\|x^{ag}_{k+1}-x^{md}_k\|^2 \right] \ \ \forall  u \in X.\label{abg_proof2}
\end{align}
Setting $u= (1-\theta_k) x^{ag}_k+\theta_k x_{k+1}$, noting \eqnok{def_x_md}, and convexity of $\|\cdot\|^2$, we have
\beqa
\|u - x^{md}_k\|^2 %= \frac{(1-\theta_k)(x^{md}_k - \eta_k x_k)}{1-\eta_k}+\theta_k x_{k+1} - x^{md}_k
&=& \left\|\frac{\eta_k(1-\theta_k)}{1-\eta_k}  (x_{k+1}-x_k) + \frac{\theta_k-\eta_k}{1-\eta_k} (x_{k+1} -x^{md}_k)\right\|^2 \nn \\
&\le& \frac{\eta_k \theta_k(1-\theta_k)}{1-\eta_k}\|x_{k+1}-x_k\|^2 + \frac{\theta_k(\theta_k-\eta_k)}{1-\eta_k} \|x_{k+1} -x^{md}_k\|^2 \nn \\
&=& \theta_k^2\left[\left(1-\frac{\theta_k \mu_f}{\mu_f+\lambda_k}\right)\|x_{k+1}-x_k\|^2 + \frac{\theta_k \mu_f}{\mu_f+\lambda_k} \|x_{k+1} -x^{md}_k\|^2\right].
\label{cvx_com}
\eeqa
Moreover, noting the smoothness of $f$ due to Lemma~\ref{grad_f_error}.c), we have
\beq\label{fag_smooth}
f(x^{ag}_{k+1};y^*(x^{ag}_{k+1})) \le f(x^{md}_k;y^*(x^{md}_k)) + \langle \nabla f(x^{md}_k;y^*(x^{md}_k)),x^{ag}_{k+1}-x^{md}_k \rangle +\frac{L_f}{2}\|x^{ag}_{k+1}-x^{md}_k\|^2.
\eeq
Multiplying \eqnok{abg_proof1} by $\theta_k$, adding it up with \eqnok{abg_proof2} and \eqnok{fag_smooth} by noting \eqnok{cvx_com}, and denoting $\Delta^{md}_k \equiv \bar \nabla f(x^{md}_k;\bar y_k)-\nabla f(x^{md}_k;y^*(x^{md}_k))$, we obtain
\begin{align}
&f(x^{ag}_{k+1};y^*(x^{ag}_{k+1})) \le (1-\theta_k) [f(x^{md}_k;y^*(x^{md}_k)) + \langle \nabla f(x^{md}_k;y^*(x^{md}_k)),x^{ag}_k-x^{md}_k \rangle] \nn \\
& + \theta_k [f(x^{md}_k;y^*(x^{md}_k)) + \langle \nabla f(x^{md}_k;y^*(x^{md}_k)),x-x^{md}_k \rangle + \frac{\mu_f}{4}\|x-x^{md}_k\|^2] \nn \\
&  -\frac{\theta_k \mu_f}{4}\|x-x_{k+1}\|^2 -\frac{1}{4} \left(1-\frac{2\theta_k^2}{\alpha_k(\mu_f+\lambda_k)}\right)\left[\theta_k \mu_f\|x_{k+1}-x^{md}_k\|^2 + [(1-\theta_k)\mu_f+\lambda_k]\|x_{k+1}-x_k\|^2\right]\nn \\
&  - \frac{1}{2 \alpha_k}\left[(1-L_f \alpha_k)\|x^{ag}_{k+1}-x^{md}_k\|^2 +\|u-x^{ag}_{k+1}\|^2 \right]
 +\frac{(1-\theta_k)\mu_f+\lambda_k}{4} \left[\|x-x_k\|^2-\|x-x_{k+1}\|^2\right] \nn \\
& +\langle \Delta^{md}_k , \theta_k(x_{k+1}-x)+x^{ag}_{k+1}-u \rangle,\label{abg_proof2}
\end{align}
which together with \eqnok{alpha_lambda}, strong convexity of $f$, the fact that
\begin{align}
&\langle \Delta^{md}_k , \theta_k(x_{k+1}-x)+x^{ag}_{k+1}-u \rangle = \langle \Delta^{md}_k , \theta_k(x_{k+1}-x^{md}_k)+\theta_k(x^{md}_k-x )+x^{ag}_{k+1}-u \rangle   \nn \\
&\le \left(\frac{3\theta_k}{\mu_f}+\frac{\alpha_k}{2}\right) \|\Delta^{md}_k\|^2   + \frac{\theta_k \mu_f}{4}\left[\frac{1}{2}\|x_{k+1}-x^{md}_k\|^2+ \|x-x^{md}_k\|^2\right] + \frac{1}{2\alpha_k}\|u-x^{ag}_{k+1}\|^2,\nn
\end{align}
imply that
\beqa
f(x^{ag}_{k+1};y^*(x^{ag}_{k+1}))&\le& (1-\theta_k) f(x^{ag}_k;y^*(x^{ag}_k)) + \theta_k f(x) +\frac{(1-\theta_k)\mu_f+\lambda_k}{4}\|x-x_k\|^2 \nn \\
&-& \frac{\lambda_k+ \mu_f}{4}\|x-x_{k+1}\|^2+\frac{6\theta_k \mu_f \alpha_k}{2\mu_f}\|\Delta^{md}_k\|^2.\label{strong_recur_acc}
\eeqa
Letting $x= x^*$ in the above inequality, subtracting $f(x^*;y^*(x^*))$ form both sides, noting \eqnok{alpha_lambda}, \eqnok{def_gama2}, and after re-arranging the terms, we obtain
\[
e_{k+1} \le (1-\gamma_k) e_k +\frac{\lambda_k}{4}\left[\|x-x_k\|^2-\|x-x_{k+1}\|^2 \right] + \frac{6\theta_k +\mu_f \alpha_k}{2\mu_f}\|\Delta^{md}_k\|^2,
\]
where $e_k= f(x^{ag}_k;y^*(x^{ag}_k))- f(x^*;y^*(x^*))+ \frac{\mu_f}{4}\|x-x_k\|^2$. Divining both sides of the above inequality by $\Gamma_{k+1}$, noting \eqnok{alpha_lambda1}, \eqnok{y_cnvrg}, \eqnok{def_grad_error}, and summing them up, we obtain \eqnok{abg_strng_cvx}.\\
We now show part b). If $f$ is convex, then setting $\mu_f=0$ in \eqnok{abg_proof2} and similar to \eqnok{strong_recur_acc}, we obtain
\beqa
f(x^{ag}_{k+1};y^*(x^{ag}_{k+1})) &\le& (1-\theta_k) f(x^{ag}_k;y^*(x^{ag}_k))+\theta_k f(x;y^*(x))+\frac{\lambda_k}{4} \left[\|x-x_k\|^2-\|x-x_{k+1}\|^2\right] \nn \\
&+& \theta_k \|x-x_{k+1}\| \|\Delta^{md}_k\|+ \frac{\alpha_k}{2} \|\Delta^{md}_k\|^2.\label{cvx_recur_acc}
\eeqa
Noting \eqnok{alpha_lambda1} and boundedness of $X$, \eqnok{abg_cvx} follows similarly to part a).
\end{proof}

\vgap

{\color{black}In the next result, we specialize the rates of convergence of Algorithm~\ref{alg_abg} by properly choosing the algorithm parameters.}

\begin{corollary} \label{corl_main_abg}
Suppose that $\{\bar y_k, x_k, x^{md}_k, x^{ag}_k \}_{k \ge 0}$ is generated by Algorithm~\ref{alg_abg}, Assumptions~\ref{f_assumption} and ~\ref{g_assumption} hold, stepsizes are chosen according to \eqnok{alpha_beta}, \eqnok{alpha_beta1}, and
\beq \label{alpha_lambda2}
\lambda_k = \frac{8 \Gamma_{k+1}}{\alpha_k} \qquad \forall k \ge 0.
\eeq
\begin{itemize}
\item [a)] If $f$ is strongly convex with parameter $\mu_f>0$, $t_k =k+1$, and
\beq \label{alpha_lambda3}
\theta_k^2 = \frac{\alpha_k \mu_f}{4}+\bar \Gamma_{k+1},
\eeq
where $\bar \Gamma_k = \Gamma_k$ with the choice of $\gamma_k =\theta_k$, we have, for any $N \ge 1$,
\begin{align}
f(x_N;y^*(x_N))-f^* &\le (1-\gamma)^N \left[f(x_0;y^*(x_0))-f^* + \frac{(\mu_f+12L_f)\|x^*-x_0\|^2}{4}+\frac{7(Q_g-1) M^2 C^2}{4 \mu_f} \right],\label{abg_strng_cvx2}
\end{align}
where
\beq \label{def_param2}
\gamma_k = \gamma = \min \left(\frac12 \sqrt{\frac{\mu_f}{3L_f}},\frac{2}{Q_g+1}\right) \quad \forall k \ge 0.
\eeq

\item [b)] If $f$ is convex, $X$ is bounded, $t_k =\sqrt{k+1}$, and
\beq \label{alpha_lambda4}
\gamma_k = \theta_k = \frac{2}{k+2} \qquad \forall k \ge 0,
\eeq
we have
\begin{align}
f(x_N;y^*(x_N))-f^* &\le \frac{2}{N(N+1)} \left[15 L_f D_X^2 + \frac{16 (Q_g-1)^2(Q_g+1)^6 C^2 M^2}{L_f}\right].\label{abg_cvx2}
\end{align}

\end{itemize}
\end{corollary}

\vgap

\begin{proof}
First, we show that the stepsizes are well-defined. Observe that by by \eqnok{def_Gamma}, \eqnok{alpha_lambda2}, \eqnok{alpha_lambda3}, and \eqnok{def_param2}, we have
\[
\gamma_k \le \theta_k, \qquad \bar \Gamma_{k+1} \le \Gamma_{k+1}, \qquad \frac{\lambda_k}{\Gamma_{k+1}} = \frac{4}{\alpha_k} \qquad \forall k \ge 0
\]
which ensures conditions \eqnok{alpha_lambda} and \eqnok{alpha_lambda1} due to the choice of $\alpha_k$ in \eqnok{alpha_beta1}. Moreover, we have
$\theta_k^2 = \alpha_k \mu_f/4 + (1-\theta_k) \bar \Gamma_k \ \ \forall k \ge 0$, which implies that
\beq \label{theta_k}
\theta_k  = \frac{-\bar \Gamma_k+\sqrt{\bar \Gamma_k^2 +4\bar \Gamma_k+\alpha_k \mu_f}}{2} \qquad \forall k \ge 0,
\eeq
and $\theta_k \in (0,1)$. Now, noting \eqnok{gamma_ineq}, \eqnok{def_param2}, the fact that $6\theta_k +\alpha_k \mu_f \le 7$, and \eqnok{abg_strng_cvx}, we obtain \eqnok{abg_strng_cvx2}.

Second, noting \eqnok{def_Gamma} and \eqnok{alpha_lambda4}, for any $\rho \in (0,1)$, we have
\begin{align}
& \Gamma_N = \frac{2}{N(N+1)} \qquad \forall N \ge 1,\nn \\
& \sum_{k=0}^{N-1} \frac{1}{\Gamma_{k+1}} \left(\theta_k \rho^{\sqrt{k+1}}+\frac{\alpha_k}{2} \rho^{2\sqrt{k+1}} \right) = \sum_{k=1}^{N} \left( k \rho^{\sqrt{k}} + \frac{k(k+1)}{12 L_f}\rho^{2\sqrt{k}} \right)
 \le \frac{288 \rho}{(1-\rho^2)^4}+\frac{840 \rho^2}{L_f(1-\rho^2)^6},
\end{align}
which together with \eqnok{abg_cvx}, imply \eqnok{abg_cvx2}.
\end{proof}

\vgap

Note that \eqnok{abg_strng_cvx2} implies that when $f$ is strongly convex, Algorithm~\ref{alg_abg} can slightly improve the iteration complexity bounds of Algorithm~\ref{alg_BG} in \eqnok{complex_bnd_strcvx} to
\beq \label{complex_bnd_strcvx_acc}
GC(f,\epsilon) = HC(g,\epsilon) = \max\left\{\sqrt{\frac{L_f}{\mu_f}},\frac{L_g}{\mu_g}\right\} {\cal O} \left(\log \frac{1}{\epsilon}\right), \qquad GC(g,\epsilon)= GC(f,\epsilon)^2.
\eeq
Moreover, \eqnok{abg_cvx2} implies that the accelerated variant of Algorithm~\ref{alg_BG} can significantly improve its complexity bounds in \eqnok{complex_bnd_cvx} to
\beq \label{complex_bnd_cvx_acc}
GC(f,\epsilon) = HC(g,\epsilon) = \max\left\{\sqrt{L_f D_X^2},\frac{Q_g^4 C M}{\sqrt{L_f}}\right\} {\cal O} \left(\frac{1}{\sqrt \epsilon}\right), \qquad GC(g,\epsilon)= GC(f,\epsilon)^{\tfrac{3}{2}},
\eeq
when $f$ is only convex.  It should be mentioned that since the acceleration does not change the complexity bounds in \eqnok{complex_bnd_nocvx} when $f$ is possibly nonconvex and hence we do not repeat that result.

\section{Stochastic Approximation Methods for Bilevel Programming}\label{sab_section}
In this section, we study the bilevel programming problem under the stochastic setting. In particular, we consider the problem of the form \eqnok{main_prob_st} and suppose that Assumption~\ref{stochastic_assumption} holds. To do so, we first need to compute a stochastic approximation of $\left[\nabla^2_{yy} g\right]^{-1}$. Noting the following well-known result about matrices, we can provide a subroutine to approximate inverse of the Hessians of the inner expectation function in \eqnok{main_prob_st}.

\begin{lemma}\label{lemma_invs}
Let $A$ be a symmetric positive definite matrix such that $\|A\| < 1$. Then, we have\\
\[
A^{-1} = \sum\limits_{i=0}^\infty [I-A]^i.
\]
\end{lemma}

\vgap

Noting the above result, we can provide the following subroutine to compute an approximation for $\left[\nabla^2_{yy} g\right]^{-1}$. It should be mentioned that such approximation has been previously used in different forms (see e.g., \cite{AgrBuHa16}).

\begin{algorithm} [H]
	\caption{The Hessian Inverse Approximation (HIA) Subroutine}
	\label{alg_HIA}
	\begin{algorithmic}

\STATE Input:
$\bar x \in X$, $\bar y \in \bbr^m$, smoothness parameter $L_g$, and positive integer number $b$.

\STATE Choose $p \in \{0,\ldots,b-1\}$ randomly.\\
{\bf For $i=1,\ldots$, $p$:}

{\addtolength{\leftskip}{0.2in}

\STATE Compute Hessian approximations $\nabla^2_{yy} G_i \equiv \nabla^2_{yy} G(\bar x, \bar y , \zeta_{i})$, where $\{\zeta_{i}\}_{i \ge 0}$ are i.i.d samples from $\zeta$.\\
}

{\bf End}
\STATE Set
\beq \label{def_Hyy}
H_{yy} \equiv H_{yy}(\bar x, \bar y, \zeta_{[r]}) = \left\{
\begin{array}{ll}
\tfrac{b}{L_g} \prod\limits_{i=1}^p \left[I - \tfrac{1}{L_g}\nabla^2_{yy} G_i \right] , & p \ge 1,\\
\tfrac{b}{L_g} I, & p = 0.
\end{array} \right.
\eeq
	\end{algorithmic}
\end{algorithm}

In the next result, we evaluate the quality of the above Hessian inverse approximation.

\begin{lemma}\label{lemma_HIA}
Let $H_{yy}$ be the output of the HIA subroutine. Then, under Assumptions~\ref{g_assumption} and ~\ref{stochastic_assumption}, we have
\beqa
%\bbe[H_{yy}] \equiv \bbe[H_{yy}(\bar x, \bar y, \zeta_{[r]})] &=& \frac{1}{L_g} \sum_{i=0}^{b-1} \left[I- \tfrac{1}{L_g}\nabla^2_{yy} g (\bar x, \bar y)\right]^i, \nn \\
\|[\nabla^2_{yy} g (\bar x, \bar y)]^{-1} - \bbe[H_{yy}]\| &\le& \frac{1}{\mu_g} \left(\frac{Q_g-1}{Q_g} \right)^b,\nn \\
\bbe[\|[\nabla^2_{yy} g (\bar x, \bar y)]^{-1}-H_{yy}\|] &\le& \frac{2}{\mu_g}, \label{Hessian_Invs_ineq}
\eeqa
where the expectation is taken with respect to both $p$ and $\zeta$.
\end{lemma}

\begin{proof}
First, note that by \eqnok{def_Hyy}, independency of $p$ and $\zeta$, and under Assumption~\ref{stochastic_assumption}, we have
\beqa
\bbe[H_{yy}] &=& \bbe_p \left[\bbe_\zeta[H_{yy}(\bar x, \bar y, \zeta_{[p]})]\right] = \tfrac{b}{L_g}\bbe_r \left[\prod\limits_{i=1}^p \left[I - \tfrac{1}{L_g}\bbe_\zeta[\nabla^2_{yy} G(\bar x, \bar y , \zeta_{i})\right]\right] \nn \\
&=& \tfrac{b}{L_g} \bbe_r \left[I - \tfrac{1}{L_g} \nabla^2_{yy} g(\bar x, \bar y)\right]^p =
\tfrac{1}{L_g}\sum_{i=0}^{b-1} \left[I - \tfrac{1}{L_g}\nabla^2_{yy} g(\bar x, \bar y)\right]^i,\label{lem_st_p1}
\eeqa
where the last equality follows from the uniform distribution of $p$. Second, noting the fact that $I \succeq \tfrac{1}{L_g} \nabla^2_{yy} g \succeq \tfrac{\mu_g}{L_g}$ due to Assumptions~\ref{g_assumption}.b), ~\ref{g_assumption}.c), and in the view of Lemma~\ref{lemma_invs}, we have
\[
[\nabla^2_{yy} g (\bar x, \bar y)]^{-1} =  \tfrac{1}{L_g} \sum_{i=0}^\infty \left[I - \tfrac{1}{L_g}\nabla^2_{yy} g(\bar x, \bar y)\right]^i
= \bbe[H_{yy}] +  \tfrac{1}{L_g}\sum_{i=b}^\infty \left[I - \tfrac{1}{L_g}\nabla^2_{yy} g(\bar x, \bar y)\right]^i,
\]
which implies that
\begin{align}
& \left\|[\nabla^2_{yy} g (\bar x, \bar y)]^{-1} - \bbe[H_{yy}]\right\| \le \tfrac{1}{L_g}\sum_{i=b}^\infty \left\|I - \tfrac{1}{L_g}\nabla^2_{yy} g(\bar x, \bar y)\right\|^i \le \frac{1}{\mu_g} \left(1- \frac{\mu_g}{L_g} \right)^b,\nn \\
&\bbe[\|H_{yy}\|] \le \tfrac{b}{L_g}\bbe \left[\prod\limits_{i=1}^r \|I - \tfrac{1}{L_g}\nabla^2_{yy} G(\bar x, \bar y , \zeta_{i})\|\right] = \tfrac{b}{L_g} \bbe_r \left[1- \tfrac{\mu_g}{L_g}\right]^r=
\tfrac{1}{L_g}\sum_{i=0}^{b-1} \left[1- \tfrac{\mu_g}{L_g}\right]^i \le \frac{1}{\mu_g}.\nn
\end{align}
Hence, \eqnok{Hessian_Invs_ineq} follows immediately in the view of condition number of $g$ and the triangle inequality for the norms .
\end{proof}

\vgap

Note that the above result show that $H_{yy}$ as the output the subroutine HIA is a biased estimation for $\left[\nabla^2_{yy} g\right]^{-1}$ with bounded variance and the biased term can be decreased by taking more samples of the Hessian. We are now ready to present a stochastic variant of Algorithm~\ref{alg_BG} for solving problem \eqnok{main_prob_st}.

\begin{algorithm} [H]
	\caption{The Bilevel Stochastic Approximation (BSA) Method}
	\label{alg_BSA}
	\begin{algorithmic}

\STATE Input:
$x_0 \in X$, $y_0 \in \bbr^m$ nonnegative sequences $\{\alpha_k\}_{k \ge 0}$, $\{\beta_t\}_{t \ge 0}$, and integer sequences $\{t_k\}_{k \ge 0}$ and $\{b_k\}_{k \ge 0}$.

\STATE Set $\bar y_0 =y_0$.

{\bf For $k=0,1,\ldots$:}

\vgap

{\addtolength{\leftskip}{0.2in}

{\bf For $t=0,1,\ldots, t_k-1$:}

}
\vgap
{\color{black}
{\addtolength{\leftskip}{0.4in}

\STATE Call the stochastic oracle of function $g$ to compute its stochastic partial gradient $G(x_k,y_t,\zeta^{(1)}_t)$ and set
\beq \label{def_yt}
y_{t+1} = y_t- \beta_t \nabla_y G(x_k,y_t,\zeta^{(1)}_t).
\eeq

}

{\addtolength{\leftskip}{0.2in}
{\bf End}

\STATE  Set $\bar y_k=y_{t_k}$. Call both stochastic oracles of functions $f$ and $g$ to compute the stochastic gradient approximation of $f$ given by
\beqa
\tilde \nabla f(x_k;\bar y_k,\w_k) &:=& \nabla_x F(x_k;\bar y_k,\xi_k)- \tilde M(x_k, \bar y_k)\nabla_y F(x_k;\bar y_k,\xi_k),\label{grad_f_st}\\
\tilde M(x_k, \bar y_k) &:=& \nabla_{xy}^2 G(x_k,\bar y_k,\zeta_k^{(2)})H_{yy}(x_k,\bar y_k,\zeta_k^{(3}),\nonumber
\eeqa
where $\w_k =(\xi_k,\zeta^{(1)}_{[t_k]}, \zeta^{(2)}_k, \zeta^{(3)}_k)$ and $H_{yy}$ is computed according to \eqnok{def_Hyy} in Algorithm~\ref{alg_HIA} with $b=b_k$.
Set
\beq \label{def_xk}
x_{k+1} = \arg\min_{u \in X} \left\{\langle \tilde \nabla f(x_k;\bar y_k,\w_k),u \rangle + \frac{1}{2 \alpha_k}\|u-x_k\|^2 \right\},
\eeq

}
}

{\bf End}
	\end{algorithmic}
\end{algorithm}

\vgap

Note that while the above algorithm has the same framework as of Algorithm~\ref{alg_BG}, it has two major differences. First, its inner loop runs essentially a stochastic gradient method to approximately find a solution to the inner problem of \eqnok{main_prob_st}. Second, to compute the gradient estimation in \eqnok{grad_f_st}, we need to estimate the Hessian inverse as well. Hence, more parameters should be appropriately chosen to establish the rate of convergence of the algorithm. To do so, we first present the well-known convergence result of of the inner loop in Algorithm~\ref{alg_BSA} as a variant of the stochastic gradient method.

\begin{lemma}\label{lemma_sgd}
Let $\{y_t\}_{t=0}^{t_k}$ be the sequence generated at the $k$-th iteration of Algorithm~\ref{alg_BSA}. If $\beta_t = 1/[\mu_g(t+2)] \ \ t \ge 0$, then we have
\beq\label{y_cnvrg_st}
\|y_{t_k} - y^*(x_k) \| \le \sqrt{\frac{2}{t_k+2}} \max\left\{\|y_0 - y^*(x_k) \|,\frac{\sigma_{g_{yy}}}{\mu_g}\right\}:=\bar A_k.
\eeq
\end{lemma}

\vgap

We now present the main convergence results for Algorithm~\ref{alg_BSA}.

\begin{theorem}[Convergence results for the BSA algorithm]\label{theom_main_bsa}
Suppose that $\{\bar y_k, x_k \}_{k \ge 0}$ is generated by Algorithm~\ref{alg_BSA}, Assumptions~\ref{f_assumption}, ~\ref{g_assumption}, ~\ref{stochastic_assumption} hold, and stepsizes are chosen such that
\beq\label{alpha_beta_st}
\beta_t = \frac{1}{\mu_g(t+2)} \quad \forall t.
\eeq
\begin{itemize}
\item [a)] Assume that $f$ is strongly convex with parameter $\mu_f>0$, there exists $C_{f_x}>0$ such that for any $\bar x \in X$ and $\bar y \in \bbr^m$, $\|\nabla_x f(\bar x;\bar y)\| \le C_{f_x}$. Then for any $N \ge 1$, we have
\begin{align}
&\bbe[f(\hat x_N;y^*(\hat x_N))]-f^* \le \Gamma_N \sum_{k=0}^{N-1} \frac{\gamma_k}{\Gamma_{k+1}} \left[\frac{1}{2 \alpha_k}\left[\|x^*-x_k\|^2-\|x^*-x_{k+1}\|^2\right]+ \alpha_k\sigma_f^2 + L_f \bar C^2 \alpha_k^2 \right. \nn \\
& \qquad \qquad \qquad \qquad \qquad \qquad + \left. \left(\frac{(1+\alpha_k \mu_f)}{\mu_f} +2L_f \alpha_k^2 \right)C^2\bar A_k^2  +\left(\frac{1}{\mu_f}+L_f \alpha_k^2\right)\hat A_k^2\right], \label{bg_strcvx_st}
\end{align}
where $\bar C$, $\bar A_k$, $\hat A_k$, $\sigma_f$, $\Gamma_k$ are, respectively, defined in \eqnok{def_Mbar}, \eqnok{y_cnvrg_st}, \eqnok{def_hatAk}, \eqnok{delta_expec}, \eqnok{def_Gamma}, and
\beq \label{def_xave2}
\hat x_N = \Gamma_N \sum_{k=1}^N \frac{\gamma_{k-1} x_k}{\Gamma_k}
\eeq
for some $\{\gamma_k\}_{k \ge 1} \in (0,1)$ with $\gamma_0=1$.
\item [b)] If $f$ is convex, $X$ is bounded, and
\beq\label{alpha_beta_st01}
\alpha_k \le \frac{1}{2 L_f}  \quad \forall k \ge 0,
\eeq
we have
\beq \label{bg_cvx_st}
\bbe[f(\bar x_N;y^*(\bar x_N))]-f^* \le \frac{1}{N} \sum_{k=0}^{N-1} \left(\frac{1}{2 \alpha_k}\left[\|x^*-x_k\|^2-\|x^*-x_{k+1}\|^2\right]+D_X (\bar A_k + \hat A_k ) + \alpha_k \sigma_f^2 \right),
\eeq
where $\bar x_N$, is defined in \eqnok{def_xave}.

\item [c)] If $f$ is possibly nonconvex, $X=\bbr^n$ (for simplicity), and \eqnok{alpha_beta_st01} holds, we have
\beq\label{bg_nocvx}
\sum_{k=0}^{N-1} \frac{\alpha_k}{2} (1-2L_f \alpha_k) \bbe[\|\nabla f(x_k;y^*(x_k))\|^2] \le f(x_0)-f^*+ \sum_{k=0}^{N-1} \left[\alpha_k \left(\bar A_k^2+\hat A_k^2\right)+
L_f \alpha_k^2 \left(\bar A_k^2+\sigma_f^2\right) \right].
\eeq
\end{itemize}
\end{theorem}

\begin{proof}
We first show part a). Denoting $\tilde \Delta_k \equiv \tilde \nabla f(x_k;\bar y_k,\w_k)-\nabla f(x_k;y^*(x_k))$ and $\delta_k =\tilde \nabla f(x_k;\bar y_k,\w_k)-\bar \nabla f(x_k;\bar y_k)$, we have $\tilde \Delta_k = \delta_k +\Delta_k$, where $\Delta_k$ is defined in the proof of Theorem~\ref{theom_main_bg}. Hence, under Assumptions~\ref{f_assumption}, ~\ref{g_assumption} , ~\ref{stochastic_assumption}, and in the view of Lemma~\ref{lemma_HIA}, we have
\beqa
\bbe[\delta_k] &=& \bbe[\nabla_x F(x_k;\bar y_k,\xi_k) - \nabla_x f(x_k;\bar y_k)]\nn \\
&+& \bbe\left[\nabla^2_{xy} G(x_k;\bar y_k,\zeta^{(2)}_k)H_{yy}(\bar x,\bar y,\zeta^{(3)}_k)\nabla_y F(x_k;\bar y_k,\xi_k) - \nabla^2_{xy} g(x_k;\bar y_k)[\nabla^2_{yy} g(x_k;\bar y_k)]^{-1}\nabla_y f(x_k;\bar y_k)\right]\nn \\
&=& \nabla^2_{xy} g(x_k;\bar y_k)\left(\bbe[H_{yy}(\bar x,\bar y,\zeta^{(3)}_k)]-\nabla^2_{yy} g(x_k;\bar y_k)]^{-1}\right)\nabla_y f(x_k;\bar y_k) \nn \\
&=& \nabla^2_{xy} g(x_k;\bar y_k)B_k\nabla_y f(x_k;\bar y_k), \nn \\
\|\bbe[\delta_k]\| &\le& \frac{C_{g_{xy}} C_{f_y}}{\mu_g} \left(\frac{Q_g-1}{Q_g} \right)^{b_k}:= \hat A_k,\label{def_hatAk} \\
\bbe[\|\delta_k\|^2] &\le& 2\sigma^2_{f_x} + \frac{4}{\mu^2_g} \left(C^2_{g_{xy}}\sigma^2_{f_y}+2 C^2_{f_y}(\sigma^2_{g_{xy}}+2C_{g_{xy}}\right):=\sigma_f^2.
\label{delta_expec}
\eeqa
Similar to \eqnok{BG_proof1} and by setting $u=x^*$, we obtain
\begin{align}
&f(x_{k+1};y^*(x_{k+1}))\le f(x_k;y^*(x_k)) + \langle \nabla f(x_k;y^*(x_k)),x^*-x_k \rangle +\frac{1}{2 \alpha_k}\left[\|x^*-x_k\|^2-\|x^*-x_{k+1}\|^2\right] \nn \\
&\qquad \qquad \qquad \qquad \qquad \qquad \qquad - \frac{(1-L_f \alpha_k)}{2 \alpha_k} \|x_{k+1}-x_k\|^2
 +  \langle \tilde \Delta_k ,x^*-x_{k+1} \rangle. \label{bsa_proof1}
\end{align}
Multiplying both sides by $2\alpha_k$, noting strong convexity of $f$, and re-arranging the terms, we have
\beq \label{bg_strcvx_seq_st0}
(1+\alpha_k \mu_f) \|x^*-x_{k+1}\|^2 \le (1-\alpha_k \mu_f) \|x^*-x_k\|^2 - (1-L_f \alpha_k) \|x_{k+1}-x_k\|^2 + 2\alpha_k \langle \tilde \Delta_k ,x^*-x_{k+1}. \rangle
\eeq
Moreover, observe that
\begin{align}
&\|\nabla f(x_k;y^*(x_k))\| \le C_{f_x}+ \frac{C_{g_{xy}}C_{f_y}}{\mu_g} : = \bar  C, \label{def_Mbar} \\
&\bbe[\|x_{k+1}-x_k\|^2] \le \alpha_k^2 \bbe[\|\tilde \nabla f(x_k;\bar y_k,\w_k)\|^2]
\le 2\alpha_k^2 \left(2\bbe[\|\Delta_k\|^2]+2\bbe[\|\delta_k\|^2]+\|\nabla f(x_k;y^*(x_k))\|^2 \right), \nn \\
& \bbe[\langle \tilde \Delta_k ,x^*-x_{k+1} \rangle] = \langle \Delta_k ,x^*-x_{k+1} \rangle + \langle \bbe[\delta_k] ,x^*-x_k \rangle + \bbe[\langle \delta_k ,x_k-x_{k+1} \rangle] \nn \\
&\le \frac{1}{2\mu_f} \left[\|\Delta_k\|^2 + 2\|\bbe[\delta_k]\|^2 \right]+ \frac{\mu_f}{4}\left[2\|x^*-x_{k+1}\|^2 +\|x^*-x_k\|^2 \right]+\frac{\alpha_k}{2} \bbe[\|\delta_k\|^2] + \frac{1}{2\alpha_k} \|x_{k+1}-x_k\|^2, \label{bnd_delta_st}
\end{align}
where the first inequality follows from \eqnok{grad_f} and boundedness assumptions on the partial derivative of $f$ and $g$, the second inequality follows from the fact that the Euclidean projection is non-expansive, and the last inequality follows from Cauchy-Schwarz inequality. Combining the above observations with \eqnok{def_hatAk}, \eqnok{delta_expec}, and \eqnok{bg_strcvx_seq_st0}, we obtain
\beq \label{bg_strcvx_seq_st}
\|x^*-x_{k+1}\|^2 \le \left(1-\frac{\alpha_k \mu_f}{2}\right) \|x^*-x_k\|^2 + \frac{\alpha_k}{\mu_f} \left[C^2 \bar A_k^2 + 2\hat A_k^2 \right]+\alpha_k^2 \sigma_f^2+2 L_f \alpha_k^3 \left(2C^2 \bar A_k^2+2\hat A_k^2+\bar M^2 \right) .
\eeq
Similarly, we  obtain
\beqa
\bbe[f(x_{k+1};y^*(x_{k+1}))]- f^* &\le& \frac{1}{2 \alpha_k}\left[\|x^*-x_k\|^2-\|x^*-x_{k+1}\|^2\right] + \frac{1}{\mu_f}(C^2 \bar A_k^2 + \hat A_k^2 )+\alpha_k (\sigma_f^2+C^2 \bar A_k^2)\nn \\
&+& L_f \alpha_k^2 \left(2C^2 \bar A_k^2+2\hat A_k^2+\bar M^2 \right).
\eeqa
Multiplying both sides by $\tfrac{\gamma_k \Gamma_N}{\Gamma_{k+1}}$ for some $\{\gamma_k\}_{k \ge 1} \in (0,1)$ with $\gamma_0=1$, summing them up, noting the fact that $\Gamma_N \sum_{k=0}^{N-1} \tfrac{\gamma_k}{\Gamma_{k+1}} = 1$ due to \eqnok{def_Gamma}, (strong) convexity of $f$, and in the view of \eqnok{def_xave2}, we obtain \eqnok{bg_strcvx_st}.

We now show part b). Observe that by boundedness of $X$ and similar to \eqnok{bnd_delta_st}, we obtain
\[
\bbe[\langle \tilde \Delta_k ,x^*-x_{k+1} \rangle] \le D_X (\bar A_k + \hat A_k ) + \alpha_k \sigma_f^2 + \frac{1}{4 \alpha_k} \|x_{k+1}-x_k\|^2.
\]
Hence, taking expectation form both sides of \eqnok{bsa_proof1}, noting the above observation, convexity of $f$, \eqnok{alpha_beta_st},
and after re-arranging the terms, we obtain
\[
\bbe[f(x_{k+1};y^*(x_{k+1}))]-f^* \le \frac{1}{2 \alpha_k}\left[\|x^*-x_k\|^2-\|x^*-x_{k+1}\|^2\right]+D_X (\bar A_k + \hat A_k ) + \alpha_k \sigma_f^2.
\]
Summing up both sides of the above inequality, diving them by $N$, and noting \eqnok{def_xave}, we obtain \eqnok{bg_cvx_st}.

To show part c), note that if $f$ is nonconvex and $X=\bbr^n$, then similar to \eqnok{nocvx_p1} and \eqnok{bnd_delta_st}, we obtain
\beqa
\bbe[f(x_{k+1};y^*(x_{k+1}))] &\le& \bbe[f(x_k;y^*(x_k))] -\frac{\alpha_k}{2}(1-2L_f \alpha_k) \|\nabla f(x_k;y^*(x_k))\|^2 +\alpha_k \left(\|\Delta_k\|^2+\|\bbe[\delta_k]\|^2\right)\nn \\
&& \qquad \qquad \qquad \qquad \qquad \qquad \qquad \qquad \qquad \qquad + L_f \alpha_k^2 \left(\|\Delta_k\|^2+\bbe[\|\delta_k\|^2]\right)\nn \\
&\le& \bbe[f(x_k;y^*(x_k))] -\frac{\alpha_k}{2}(1-2L_f \alpha_k) \|\nabla f(x_k;y^*(x_k))\|^2 +\alpha_k \left(\bar A_k^2+\hat A_k^2\right)\nn \\
&&\qquad \qquad \qquad \qquad \qquad \qquad \qquad \qquad \qquad \qquad + L_f \alpha_k^2 \left(\bar A_k^2+\sigma_f^2\right).\nn
\eeqa
Rest of the proof is similar to that of Theorem~\ref{theom_main_bsa} and hence, we skip the details.
\end{proof}

\vgap

In the next result, we specialize rates of convergence of Algorithm~\ref{alg_BSA} when applied to different class of problems given by \eqnok{main_prob_st}.

\begin{corollary} \label{lemma_main_bsa}
Suppose that $\{\bar y_k, x_k \}_{k \ge 0}$ is generated by Algorithm~\ref{alg_BSA}, Assumptions~\ref{f_assumption},~\ref{g_assumption}, and~\ref{stochastic_assumption} hold. Also, assume that $\beta_k$ is set to \eqnok{alpha_beta_st}.
\begin{itemize}
\item [a)] Assume that $f$ is strongly convex with parameter $\mu_f>0$, $t_k =k$, $b_k = | \lceil \tfrac{1}{2} \log_{1-1/Q_g} k+2 \rceil |$, and
\beq \label{alpha_beta1_st}
\alpha_k = \frac{4}{\mu_f (k+2)} \ \ \forall k \ge 0.
\eeq
Then, for any $N \ge 1$, we have
\begin{align}\label{bg_strcvx_st1}
&\bbe[f(\hat x_N;y^*(\hat x_N))]-f^* \le \frac{\mu_f}{2N(N+1)}\|x^*-x_0\|^2+ \frac{2}{\mu_f^2(N+1)}\Big[C_3^2 + (8L_f+3\mu_f)C_1^2 \nn \\
&\qquad \qquad \qquad  \qquad \qquad \qquad \qquad  +(4L_f+\mu_f)C_2^2 + 4\mu_f\sigma_f^2+ \frac{16 L_f \bar C^2 \ln N}{N}\Big].
\end{align}
where
\beqa
C_1 &=& C \max\left\{\max_{x \in X}\|y_0 - y^*(x) \|, \frac{\sigma_{g_{yy}}}{\mu_g}\right\}, \qquad  C_2 = \frac{C_{g_{xy}} C_{f_y}}{\mu_g},\nn \\
C_3^2 &=& \frac{2(32 L_f+\mu_f)(C_1^2+C_2^2)+16L_f \bar C^2}{\mu_f}+4\sigma_f^2.\label{def_param_st}
\eeqa

\item [b)] If $f$ is convex, $X$ is bounded, an iteration limit $N$ is given, $t_k = \lceil k+1 \rceil$, $b_k = | \lceil \log_{1-1/Q_g} \sqrt{k+1} \rceil |$, and
\beq \label{alpha_beta1_st2}
\alpha_k = \frac{1}{2L_f \sqrt{N+1}} \ \ \forall k = 0, 1, \ldots, N
\eeq
for any given $N \ge 1$, we have
\beq \label{bg_cvx1_st}
\bbe[f(\bar x_N;y^*(\bar x_N))]-f^* \le \frac{1}{\sqrt{N}} \left[2L_f \|x_0-x^*\|^2+3 D_X\left(\sqrt{2}M_1+M_2\right) +\frac{\sigma_f^2}{2 L_f} \right].
\eeq

\item [c)] If $f$ is possibly nonconvex, $X=\bbr^n$ (for simplicity), $\alpha_k$ is set to \eqnok{alpha_beta1_st2}, $t_k = \lceil \sqrt{k+1} \rceil$, and $b_k = | \lceil \tfrac12 \log_{1-1/Q_g} \sqrt{k+1} \rceil |$, we have
\beq\label{bg_nocvx1_st}
\bbe\left[\|\nabla f(x_R;y^*(x_R))\|^2\right] \le \frac{8}{\sqrt{N}} \left[4 L_f[f(x_0;y^*(x_0))-f^*] + 36 M_1^2+ 6 M_2^2 +\sigma_f^2 \right],
\eeq
where the expectation is taken with respect to the integer random variable $R$ uniformly distributed over $\{0,1,\ldots, N-1\}$.
\end{itemize}
\end{corollary}

\begin{proof}
First, note that by \eqnok{bg_strcvx_seq_st}, \eqnok{alpha_beta1_st}, \eqnok{def_param_st}, and choices of $t_k$ and $b_k$, we have
\[
\|x^*-x_{k+1}\|^2 \le (1-\gamma_k) \|x^*-x_k\|^2 + {\gamma_k^2 M_3^2}{\mu_f^2},
\]
where $\gamma_k =2/(k+2)$ and $M_3$ is defined in \eqnok{def_param_st}. Dividing both sides of the above inequality by $\Gamma_{k+1}$, summing them up and noting \eqnok{def_Gamma}, we obtain
\beq \label{bg_strcvx_seq_st2}
\|x^*-x_N\|^2 \le \frac{4 M_3^2}{\mu_f^2(N+1)} \qquad \forall N \ge 1.
\eeq
Second, noting the above bound on the generated sequences, \eqnok{alpha_beta1_st}, and choice of $\gamma_k$ same as above, we have
\begin{align}
&\sum_{k=0}^{N-1} \frac{\gamma_k}{2\alpha_k \Gamma_{k+1}} \left[\|x^*-x_k\|^2-\|x^*-x_{k+1}\|^2\right]\nn \\
&= \frac{\mu_f}{4}\|x^*-x_0\|^2+ \sum_{k=1}^{N-1} \left(\frac{\gamma_k}{2\alpha_k \Gamma_{k+1}} - \frac{\gamma_{k-1}}{2\alpha_{k-1} \Gamma_k} \right) \|x^*-x_k\|^2- \frac{\gamma_{N-1}}{2\alpha_N \Gamma_N}\|x^*-x_N\|^2\nn \\
&\le \frac{\mu_f}{4}\|x^*-x_0\|^2+ \frac{M_3^2 N}{\mu_f}, \nn \\
&\sum_{k=0}^{N-1} \frac{\gamma_k}{\Gamma_{k+1}} \left(\frac{(1+\alpha_k \mu_f)}{\mu_f} +2L_f \alpha_k^2 \right)C^2\bar A_k^2  \le \frac{(8L_f+3\mu_f)M_1^2 N}{\mu_f^2}, \nn \\
&\sum_{k=0}^{N-1} \frac{\gamma_k}{\Gamma_{k+1}} \left(\frac{1}{\mu_f}+L_f \alpha_k^2\right)\hat A_k^2 \le \frac{(4L_f+\mu_f)M_2^2 N}{\mu_f^2}, \nn \\
&\sum_{k=0}^{N-1} \frac{\gamma_k \alpha_k \sigma_f^2}{\Gamma_{k+1}} \le \frac{4\sigma_f^2 N}{\mu_f}, \qquad \qquad \qquad
\sum_{k=0}^{N-1} \frac{\gamma_k \alpha_k^2 L_f \bar M^2}{\Gamma_{k+1}} \le \frac{16 L_f \bar M^2 \ln N}{\mu_f^2},
\end{align}
which together with \eqnok{bg_strcvx_st}, imply \eqnok{bg_strcvx_st1}.

Third, noting \eqnok{alpha_beta1_st}, \eqnok{def_hatAk}, \eqnok{y_cnvrg_st}, and with the choices of $t_k =k+1$ and $b_k = | \lceil \log_{1-1/Q_g} \sqrt{k+1} \rceil |$, we have
\begin{align}
&\sum_{k=0}^{N-1} \frac{1}{2 \alpha_k}\left[\|x^*-x_k\|^2-\|x^*-x_{k+1}\|^2\right] = L_f \sqrt{N+1}\left[\|x^*-x_0\|^2-\|x^*-x_N\|^2\right] \le 2 L_f \sqrt{N} \|x^*-x_0\|^2, \nn \\
&\sum_{k=0}^{N-1} \bar A_k \le 3 \sqrt{2 N} M_1, \nn \\
&\sum_{k=0}^{N-1} \hat A_k \le M_2 \sum_{k=0}^{N-1} \frac{1}{\sqrt{k+1}} \le 3 M_2 \sqrt{N}, \qquad \qquad \sum_{k=0}^{N-1} \alpha_k \le \frac{\sqrt{N}}{2 L_f}, \nn
\end{align}
which together with \eqnok{bg_cvx_st}, imply \eqnok{bg_cvx1_st}.

Finally, noting \eqnok{alpha_beta1_st2}, the choices of $t_k = \lceil \sqrt{k+1} \rceil$, and $b_k = | \lceil 0.5 \log_{1-Q_g^{-1}} \sqrt[4]{k+1} \rceil |$, we have
\begin{align}
&\sum_{k=0}^{N-1} \frac{\alpha_k}{2} (1-2L_f \alpha_k)  \ge \frac{\sqrt{N}}{32 L_f}, \qquad \qquad \qquad
\sum_{k=0}^{N-1} \left[\alpha_k(1+L_f\alpha_k)\bar A^2_k \right] \le \frac{9 \max\{\|y_0 - y^*(x_k)\|^2, \tfrac{\sigma^2_{g_{yy}}}{\mu_g^2} \}}{L_f},\nn \\
&\sum_{k=0}^{N-1} \alpha_k \hat A^2_k \le \frac{3}{2L_f} \left(\frac{L_{g_{xy}} C_{f_y}}{\mu_g}\right)^2 \qquad \qquad \quad
\sum_{k=0}^{N-1} L_f \sigma_f^2 \alpha_k^2 \le \frac{\sigma_f^2}{4 L_f}.
\end{align}
Rest of the proof is similar to that of Lemma~\ref{lemma_main_bsa}.c).
\end{proof}

We make a few remarks about the above results. First, note that \eqnok{bg_strcvx_st1} and \eqnok{def_param_st} imply that sample complexities of Algorithm~\ref{alg_BSA} for finding an $\epsilon$ solution of problem \eqnok{main_prob_st} are bounded by
\beqa \label{complex_bnd_strcvx_st}
&&SGC(f,\epsilon) = {\cal O} \left(\sqrt{\frac{\mu_f \|x^*-x_0\|^2}{\epsilon}} + \frac{M_3^2}{\mu_f \epsilon}\right), \qquad \qquad SGC(g,\epsilon)= SGC(f,\epsilon)^2,\nn \\
&& SHC(g,\epsilon) = SGC(f,\epsilon)\log SGC(f,\epsilon),
\eeqa
when $f$ is strongly convex. Furthermore, when $f$ is only convex, the above bounds are change to
\beqa \label{complex_bnd_cvx_st}
&&SGC(f,\epsilon) = {\cal O} \left(L_f D_X^2+ \frac{M_3^2}{L_f}\right)\frac{1}{\epsilon^2}, \qquad  \qquad SGC(g,\epsilon)= SGC(f,\epsilon)^2, \nn \\ &&SHC(g,\epsilon) = SGC(f,\epsilon)\log SGC(f,\epsilon)
\eeqa
due to \eqnok{bg_cvx1_st}. Finally, when $f$ is possibly nonconvex, \eqnok{bg_nocvx1_st} implies that sample complexities of Algorithm~\ref{alg_BSA} are in the order of
\beqa \label{complex_bnd_nocvx_st}
&&SGC(f,\epsilon) = {\cal O} \left(L_f[f(x_0)-f^*]+M_3^2\right)\frac{1}{\epsilon^2}, \qquad \qquad SGC(g,\epsilon)= SGC(f,\epsilon)^{\tfrac32}, \nn \\ &&SHC(g,\epsilon) = SGC(f,\epsilon)\log SGC(f,\epsilon).
\eeqa

\vgap

{\color{black}
To the best of our knowledge the above results seem to be the first finite-sample complexity bounds for the stochastic bilevel programming problem. Moreover the above bounds in \eqnok{complex_bnd_strcvx_st} and \eqnok{complex_bnd_cvx_st} for $SGC(f,\epsilon)$ when $f$ is (strongly) convex match the well-known sample complexity results for the class of three stage stochastic optimization problem (see e.g., \cite{sn04}).
}

\section{Concluding Remarks}\label{concld_section}
We have presented iterative algorithms for solving bilevel optimization problems where the inner problem is strongly convex. Under mild assumptions on the partial derivatives of both objective function, we also provide finite-time convergence analysis of proposed algorithm and established its iteration complexity under different convexity assumptions on the outer objective function. Using an acceleration scheme, we recover (nearly) optimal iteration complexity of the single level problems for the bilevel problem. Moreover, we have developed a randomized stochastic approximation algorithm that work in the stochastic setting where both objective functions are given in the form of expectations. Convergence analysis and sample complexity bounds of this algorithm, are also provided. To the best of our knowledge, this is the first time that iterative algorithms with established iteration (sample) complexities are presented for solving bilevel optimization problems.

%\bibliographystyle{siam}
%\bibliography{../sghadimi-bib}
\newcommand{\noopsort}[1]{} \newcommand{\printfirst}[2]{#1}
  \newcommand{\singleletter}[1]{#1} \newcommand{\switchargs}[2]{#2#1}

\end{document}